\documentclass[10pt,a4paper]{amsart} 
\usepackage[draft]{optional}
\usepackage[british,american]{babel}
\usepackage[margin=0cm]{geometry}
\usepackage{lipsum}

\usepackage{mathrsfs}
\usepackage{amsfonts, amsmath, wasysym, caption}

\usepackage{amssymb,amsthm, paralist}

\usepackage{amsopn}
\usepackage{accents,bbm,bibgerm,dsfont,eucal,esint,paralist,url,verbatim,wasysym}
\usepackage[normalem]{ulem}
\usepackage{bbold}

\usepackage{
latexsym,nicefrac
}

\usepackage[usenames]{color}
\usepackage{tikz}
\usetikzlibrary{decorations.pathreplacing}
\usepackage{url}

\definecolor{darkgreen}{rgb}{0,0.5,0}
\definecolor{darkred}{rgb}{0.7,0,0}

\usepackage[colorlinks, 
citecolor=darkgreen, linkcolor=darkred
]{hyperref}

\hypersetup{
 pdfauthor={Nadine Grosse, Melanie Rupflin},
 pdftitle={Sharp eigenvalue estimates on degenerating surfaces},
 pdflang={en}
}

\usepackage{esint}
\usepackage{bibgerm}
\usepackage[normalem]{ulem}

\theoremstyle{plain}
\newtheorem{lemma}{Lemma}[section]
\newtheorem{thm}[lemma]{Theorem}
\newtheorem{prop}[lemma]{Proposition}
\newtheorem{cor}[lemma]{Corollary}
\theoremstyle{definition}
\newtheorem{defn}[lemma]{Definition}

\newtheorem{rmk}[lemma]{Remark}

\numberwithin{equation}{section}

\newcommand{\al}{\alpha}

\newcommand{\de}{\delta}

\newcommand{\Om}{\Omega}

\newcommand{\la}{\lambda}
\newcommand{\La}{\Lambda}

\newcommand{\si}{\sigma}

\newcommand{\Si}{\Sigma}
\renewcommand{\th}{\theta}
\newcommand{\Th}{\Theta}

\newcommand{\vth}{\vartheta}

\newcommand{\Ups}{\Upsilon}

\newcommand{\R}{\ensuremath{{\mathbb R}}}
\newcommand{\N}{\ensuremath{{\mathbb N}}}

\newcommand{\C}{\ensuremath{{\mathbb C}}}

\DeclareMathOperator{\inj}{inj}

\newcommand{\norm}[1]{\Vert#1\Vert}

\def\osc{\mathop{{\mathrm{osc}}}\limits}

\newcommand{\arsinh}{{\rm arsinh}}

\newcommand{\beq}{\begin{equation}}
\newcommand{\eeq}{\end{equation}}
\newcommand{\beqw}{\begin{equation*}}
\newcommand{\eeqw}{\end{equation*}}
\newcommand{\beqs}{\begin{equation*}}
\newcommand{\eeqs}{\end{equation*}}
\newcommand{\beqa}{\begin{equation}\begin{aligned}}
\newcommand{\eeqa}{\end{aligned}\end{equation}}
\newcommand{\beqas}{\begin{equation*}\begin{aligned}}
\newcommand{\eeqas}{\end{aligned}\end{equation*}}
\newcommand{\brmk}{\begin{rmk}}
\newcommand{\ermk}{\end{rmk}}
\newcommand{\partref}[1]{\hbox{(\csname @roman\endcsname{\ref{#1}})}}
\newcommand{\half}{\frac{1}{2}}

\newcommand{\define}{\mathrel{\mathrm {:=}}}

\renewcommand{\i}{\mathrm{i}}

\newcommand{\M}{\ensuremath{{\mathcal M}}_{-1}}

\newcommand{\abs}[1]{\vert#1\vert} 
\newcommand{\babs}[1]{\left| #1\right|}

\newcommand{\eps}{\varepsilon}
\newcommand{\na}{\nabla}

\newcommand{\Hol}{{\mathcal{H}}} 

\newcommand{\ran}{\rangle}
\newcommand{\lan}{\langle}
\newcommand{\Col}{\mathcal{C}}
\newcommand{\thin}{\text{-thin}}
\newcommand{\thick}{\text{-thick}}

\newcommand{\Rea}{\mathrm{Re}\,}
\newcommand{\Ima}{\mathrm{Im}\,}
\newcommand{\Cyl}{{\mathscr{C}}}

\newcommand{\Proj}{P_g^\Hol}

\newcommand{\ov}{\overline}

\newcommand{\thalf}{\tfrac12}

\newcommand{\Area}{\text{Area}}

\newcommand{\uth}{\partial_\th u_1}

\title{{\sc
Sharp eigenvalue estimates on degenerating surfaces 
}
\\ 
}
\author{Nadine Gro{\ss}e and Melanie Rupflin}

\begin{document}

\begin{abstract}
We consider the first non-zero eigenvalue $\lambda_1$ of the Laplacian on hyperbolic surfaces for which one disconnecting collar degenerates and prove that $8\pi \na\log(\lambda_1)$ essentially agrees with the dual of the differential of the degenerating Fenchel-Nielsen length coordinate. As a consequence, 
we can improve previous results of Schoen, Wolpert, Yau \cite{SWY} and Burger \cite{Burger} to obtain estimates with optimal error rates and obtain new information on the leading order terms of the polyhomogeneous expansion of $\la_1$ of Albin, Rochon and Sher \cite{ARS1}.
\end{abstract}

\maketitle

\section{Introduction and results}\label{sect:intro}
Let $M$ be a closed oriented surface of genus $\gamma\geq 2$ (always assumed to be connected) and let $g$ be a hyperbolic (i.e.~Gauss curvature $K_g\equiv -1$)  metric on $M$.
Let $\si^1$ be a simple closed geodesic in $(M,g)$ which decomposes $M$ into two connected components  $M^+$ and $M^-$.
We consider surfaces for which the length $\ell_1=L_g(\si^1)$  is small compared to the length of any other simple closed geodesic in $(M,g)$.  In this case the 
first eigenvalue of the Laplacian on $(M,g)$ turns out to be small and 
to essentially only depend on $\ell_1$ and the genera of $M^\pm$.

The asymptotic behaviour of small eigenvalues on
degenerating surfaces 
was first considered by  Schoen, Wolpert and Yau in \cite{SWY}. They studied surfaces with bounded negative curvature $-c\leq K_g\leq -\tilde c<0$ and proved in particular that if 
the collapsing geodesics decompose $M$ 
into $n+1$ connected components then 
precisely $n$ eigenvalues $0<\la_1\leq\ldots\leq\la_n$ tend to zero,
 with the rate of convergence being linear with respect to the (sum of the) lengths of the corresponding geodesics. 
Their results apply in particular to the setting of one collapsing disconnecting geodesic $\si^1$ we described above and in this case yield that 
\beq
\label{est:la_rough}
c\ell_1\leq \la_1 \leq C\ell_1 \text{ while } \la_2\geq \tilde c>0,
\eeq
for constants 
$c,\tilde c>0$  and $C<\infty$ that depend, apart from the genus, only on a lower bound on the lengths of the simple closed geodesics different from $\si^1$, or equivalently on a lower bound 
$\hat\de>0$ for the injectivity radius on $M\setminus \Col(\si^1)$. Here and in the following $\Col(\si^1)$ denotes the collar neighbourhood around $\si^1$ described by the collar lemma that we recall in Lemma~\ref{lemma:collar}.
\begin{rmk}\label{rmk:la-simple}
We note that \eqref{est:la_rough} implies in particular that $\la_1$ is simple 
provided $\ell_1=L_g(\si^1)\leq \ell_0$ for a suitably small constant $\ell_0=\ell_0(\hat\de, \gamma)\leq \ell_0(\gamma)$. 
\end{rmk}
A refined picture of the behaviour of small eigenvalues on degenerating hyperbolic surfaces was then given by Burger in \cite{Burger-kurz}
and \cite{Burger}, who compared the small eigenvalues of $-\Delta_g$ on $M$ with the eigenvalues $\widehat \la_j$ of the Laplacian of a weighted graph that is associated to the set of collapsing geodesics. In  \cite{Burger-kurz} he established that  
 $\frac{\la_j}{\widehat \la_j}\to \frac{1}{2\pi^2}$, $1\leq j\leq n$, as the surface collapses and subsequently refined this convergence result in \cite{Burger} by giving both a lower bound (of order $O(\sqrt{\ell})$) and an upper bound (of order $O(\ell\log \ell)$) on the resulting errors.
We note that in the setting we consider here his result from  \cite{Burger} yields that
\beq \label{est:burger}
C_{top}-C\sqrt{\ell_1}\leq \frac{\la_1}{\ell_1}\leq C_{top}+C\ell_1 \abs{\log(\ell_1)}\eeq
where $C_{top}$ is  given in terms of the genera $\gamma^{\pm}$ of the connected components $M^{\pm}$ of $M\setminus \si^1$
\beq \label{def:Ctop}
C_{top}=\frac{-\chi(M)}{2\pi^2 \chi(M^-)\chi(M^+)}=\frac{2(\gamma-1)}{2\pi^2 (1-2\gamma^+)\cdot (1-2\gamma^-)}.
\eeq
We remark that the upper bound in \eqref{est:burger} can be obtained directly from comparing with a function that is linear on the collar $\Col(\si^1)$ (or alternatively a function that solves the corresponding ODE on the collar) and constant on the rest of the surface while the proof of the lower bound  is far more involved and does not yield the same order of the error. 

We note that \eqref{est:burger} implies in particular that if $g$ and $\tilde g$ are two metrics which satisfy the assumptions above for geodesics $\si^1$ and $\tilde \si^1$ of the same length $\ell_1$ and connected components $M^\pm$ and $\tilde M^\pm$ of the same genera $\gamma^\pm$ then 
\beq
\label{est:diff-Burger}
\ell_1^{-1}\abs{\la_1(M,g)-\la_1(M,\tilde g)}\leq C\sqrt{\ell_1}.
\eeq
It is natural to ask whether the lower bound in \eqref{est:burger} and hence also the above estimate can be improved to $O(\ell_1\abs{\log(\ell_1)})$ and, more importantly, whether such an estimate would be optimal, respectively whether one can derive an estimate of the form \eqref{est:diff-Burger} with optimal error rates.

In the present work 
we will give positive answers to both of these questions and indeed derive both $C^0$- and $C^1$-estimates with sharp error bounds. 
Most of our analysis is quite different from the methods in \cite{Burger} as we use a dynamic approach and consider the variation of the eigenvalues induced by a change of the geometry of $(M,g)$, or to be more precise by a change of the Fenchel-Nielsen coordinates. We then obtain $C^0$-bounds, such as refinements of \eqref{est:diff-Burger} and \eqref{est:burger}, only as corollary 
of our $C^1$-bounds. 

We remark that bounds on some derivatives of small eigenvalues have been obtained previously by Batchelor \cite{Batch} who considered the change of the small eigenvalues induced by a change of the length of the collapsing geodesics, so in our case  $\frac{\partial \la_1}{\partial \ell_1}$, though his error estimates are only of order $O(\frac{1}{\abs{\log(\ell)}})$ and would thus in particular not allow for any improvement of \eqref{est:burger}.

 To state our first main result, we recall, see e.g. \cite{Buser} and \cite{Hu}, that we may extend 
any given simple closed geodesic $\si^1$ in a closed oriented hyperbolic surface $(M,g)$ to a collection $\mathcal{E}=\{\si^1,\ldots,\si^{3(\gamma-1)}\}$ of simple closed geodesics in $(M,g)$ that decompose the surface into pairs of pants. 
We also recall that we can and will choose this collection  so that the length of all geodesics $\si^j$ is bounded from above by a  constant $\bar L$ that depends only on the genus and an upper bound on $L_g(\si^1)$,
so in the situation of Remark~\ref{rmk:la-simple}, by  some $\bar L=\bar L(\gamma)$. 
The metric $g$ is determined (up to pull-back by diffeomorphisms) by the corresponding Fenchel-Nielsen coordinates $(\ell_i,\psi_i)$, for $\ell_i=L_g(\si^i)$ and $\psi_i$ the corresponding twist parameters and our first main result gives the following sharp $C^1$-bounds on the dependence of the first eigenvalue on the Fenchel-Nielsen coordinates.

\begin{thm}\label{thm:FN}
Let $(M,g)$ be a closed oriented hyperbolic surface of genus $\gamma\geq 2$ and let
$\si^1$ be a simple closed geodesic which disconnects 
$M$ into two connected components. We let  $\hat \delta>0$ be a lower bound  on the injectivity radius $\inf_{M\setminus \Col(\si^1)}\inj_g(p)$ away from the collar around $\si^1$ and suppose that $\ell_1\leq\ell_0$, for $\ell_0=\ell_0(\hat \de,\gamma)>0$ as in Remark~\ref{rmk:la-simple}. 
\newline
Let $\si^2,\ldots, \si^{3(\gamma-1)}$ be simple closed geodesics so that 
$\mathcal{E}=\{\si^1,\ldots,\si^{3(\gamma-1)}\}$ decomposes $(M,g)$ into pairs of pants, which we can furthermore assume to be chosen so that $L_g(\si^j)\leq \bar L=\bar L(\gamma)$ for every $j$.
Then the  first non-zero eigenvalue $\la_1$ of $-\Delta_g$
has the following dependence on the 
corresponding Fenchel-Nielsen length and twist coordinates
$\ell_i$ and $\psi_i$:
\newline
There exists a constant $C$ depending only on $\hat \delta $ and the genus of $M$ 
 so that
\beq\label{claim:dla-length}
\babs{\frac{\partial\lambda_1}{\partial\ell_1}- \frac{\la_1}{\ell_1}}\leq C\ell_1 \abs{\log(\ell_1)}
\text{ and }
\babs{\frac{\partial\la_1}{\partial\ell_j}}\leq C\ell_1^2  \text{ for } j\neq 1
\eeq
while a change of the twist coordinates can only change the first eigenvalue by 
\beq\label{claim:dla-twist}
\babs{\frac{\partial\la_1}{\partial\psi_1}}\leq C\ell_1^4, \qquad \text{
respectively } \qquad 
\babs{\frac{\partial\la_1}{\partial\psi_j}}\leq C\ell_1^2 \quad\text{ for } j=2,\ldots,3(\gamma-1).\eeq
\end{thm}

We note that the facts that $\la_1$ is simple and invariant under pull-back by diffeomorphisms 
guarantee that 
the above derivatives are well-defined, see also Lemma~\ref{cor:grad-la}. We will prove the above result  based on an essentially explicit characterisation of $\na \la_1$ given later in Theorem~\ref{thm:1}.  

As a consequence of the $C^1$-bounds on $\la_1$ stated in the above result we immediately obtain the following refinement of the result of Burger \cite{Burger} in the considered setting

\begin{cor}
\label{thm:C0-est}
Let $M$ be a closed oriented surface of genus $\gamma\geq 2$, $\bar\si$ be a simple closed curve that disconnects $M$ into two connected components of genera $\gamma^{\pm}$ and let $C_{top}$ be given by \eqref{def:Ctop}.
Then there exists a function $f\colon (0,2 \arsinh(1))\to \R^+$ 
that depends only on $\gamma^\pm$ and satisfies
\beqs 
\label{claim:est-f}
\babs{\frac{f(\ell_1)}{\ell_1}-C_{top}}\leq C\ell_1\abs{\log(\ell_1)}
\eeqs
and for any $\hat \delta>0$ there exists a constant $C=C(\hat \delta, \gamma)$ such that the following holds true: \newline
Let $g$ be any hyperbolic metric on $M$ for which $\inj_{g}(x)\geq \hat \delta$ on $M\setminus \Col(\si^1)$, $\si^1$ the unique geodesic in $ (M,g)$ that is homotopic to $\bar \si$, 
and for which $\ell_1\define L_g(\si^1)< 2\arsinh(1)$. Then the first eigenvalue $\la_1(M,g)$ of $-\Delta_g$ satisfies
\beq 
\label{claim:la-f}
\abs{\la_1(M,g)-f(\ell_1)}\leq C\ell_1^2.
\eeq
In particular, for metrics $g,\tilde g$ for which the lengths of the corresponding geodesics $\si^1$ and $\tilde \si^1$ agree, we have that 
\beq 
\label{est:which_is_sharp}
\abs{\la_1(M,g)-\la_1(M,\tilde g)}\leq C\ell_1^2.\eeq
\end{cor}

This result is sharp as we shall prove 
\begin{thm} \label{thm:sharp}
For every genus $\gamma\geq 2$ there exist constants $\hat\de>0$, $\bar c>0$ and $\bar \ell>0$
 so that the following holds true.
Let $M$ be a closed oriented surface of genus $\gamma$ and let $\bar\si$ be a simple closed curve that disconnects $M$ into two connected components of genera $\gamma^{\pm}$.
\newline
Then there exist families of hyperbolic metrics $(g_{\ell})_{\ell\in(0,\bar \ell)}$ and 
$(\tilde g_{\ell})_{\ell\in(0,\bar\ell)}$
 satisfying the assumptions of Corollary~\ref{thm:C0-est} for the fixed $\hat \delta>0$ and with 
 $L_{\tilde g_{\ell}}(\si^1_{\tilde g_{\ell}})=\ell=  L_{g_{\ell}}(\si^1_{ g_{\ell}})$ for which  
\beq 
\label{claim:thm-sharp}
 \abs{\la_1(M,g_{\ell})-\la_1(M,\tilde g_{\ell})} \geq \bar c \cdot \ell^2.\eeq
\end{thm}

For the proof of Theorem~\ref{thm:sharp} we will construct families of metrics satisfying the assumptions of Theorem~\ref{thm:FN} for which $\frac{\partial \la_1}{\partial \ell_2}\geq c\ell_1^2$, compare Section~\ref{subsec:proof_sharp}.

We recall that the deep results \cite{ARS1} of Albin, Rochon and Sher establish that the resolvent operator on Riemannian manifolds  has a polyhomogeneous expansion along degenerating families of metrics that are product-type $d$-metrics of order $2$, see \cite[Sec. 1.2]{ARS1} for the precise definitions. In particular, our degenerating families fit into this class, as follows from \cite{Melrose-Zhu} and \cite[Theorem 4]{Obitsu-Wolpert}. Thus, the results of \cite{ARS1, ARS2} ensure that in the situation considered here, where the first eigenvalue is simple, $\la_1$ itself admits such a polyhomogeneous expansion, which as observed \cite[Prop 7.1]{ARS2} provides an alternative way of obtaining the result of Burger that 
$$\la_1=c\ell_1+o(\ell_1).$$
While these results already established that the leading order term in the polyhomogeneous expansion 
\beq
\label{eq:polyhom}
\la_1=\sum_{i\geq 1} \sum_{k=1}^{N_i} f_{i,j} \ell_1^{\alpha_i}  \log(\ell_1)^{k}, \qquad \al_i\in \R, N_i\in \N_0, 
\eeq
is given by $C_{top}\ell_1$, with \cite{Burger} furthermore proving that the next term must appear with an exponent of at least $\al_2\geq \frac{3}{2}$, our 
results now give the following new insight into the leading order terms of this expansion:
\begin{itemize}
\item[-] the second order term in the above expansion appears with exponent $\al_2=2$
\item[-] we can have at most one logarithmic term of order $\al_2=2$, namely $f_{2,1}\ell_1^2\log(\ell_1)$, and the coefficient of this term is constant
\item[-]the next term in the expansion is $f_{2,0} \ell_1^2$ and the coefficient of this term is non-constant, in particular cannot not vanish in general.
\end{itemize}
It would be of interest to understand whether the logarithmic term $f_{2,1}\ell_1^2\log(\ell_1)$ is non-zero which would mean that the \textit{two leading order terms} of the expansion depend only on the genus of the surface, or whether conversely this term is zero, which would mean that the first two terms in the polyhomogeneous expansion are indeed polynomial in $\ell_1$. 

We also note that the results of Schoen, Wolpert and Yau \cite{SWY}, Burger \cite{Burger} and Batchelor \cite{Batch} apply to more general settings of several degenerating collars, as do the results on holomorphic quadratic differentials from \cite{holo-paper} that we use in our proof and that the refined analysis of small eigenvalues in this more general setting will be addressed in future work. 

We remark that the study of eigenvalues of the Laplacian on manifolds has a long and fruitful history. We recall in particular that 
  the work of Cheeger \cite{Cheeger} establishes that the first eigenvalue of the Laplacian on any manifold is bounded from below by $\frac{1}{4}h^2(M,g)$, 
  while  Buser \cite{Buser-upper} obtained an upper bound on $\la_1$ of $2\sqrt{K}(\dim(M)-1)h(M,g)+10h^2(M,g)$, $-(\dim(M)-1)K$ a lower bound on the Ricci-curvature and $h(M,g)$ the Cheeger isoperimetric constant, compare also \cite{Ledoux}. 
 Properties of eigenvalues on Riemannian manifolds in general, and hyperbolic surfaces in particular,  and their relations to other topics  such as Selberg's eigenvalue conjecture (see e.g. \cite{Sar}) and minimal surfaces (see \cite{Fraser-Schoen}), have been considered by many authors. We refer in particular to the books of Buser \cite{Buser} and Bergeron \cite{Berg} for an overview of results on eigenvalues on hyperbolic surfaces and note that the asymptotic behaviour of small eigenvalues has been considered also by  Grotowski, Huntley and Jorgenson in \cite{GHJ}, and in a generalised setting by Judge \cite{Judge},  that Colbois and Colin de Verdi\`ere used the study of eigenvalues on weighted graphs to obtain multiplicity results for eigenvalues 
on hyperbolic surfaces \cite{CC} and that the question of how many eigenvalues of $-\Delta_g$ on a hyperbolic surface of genus $\gamma$ can be smaller than $\frac14$ has been addressed in particular by \cite{Buser_77}, \cite{Schmutz} and  \cite{Otal-Rosas}.

This paper is structured as follows: We will begin by recalling the necessary background material on hyperbolic surfaces and holomorphic quadratic differentials in Section~\ref{sect:background}. The proof of our main results are then all given in Section~\ref{sect:main}: There we first prove an essentially explicit characterisation of the $L^2$-gradient of $\la_1$ that seems to be of independent interest, see Theorem~\ref{thm:1}. We then use this theorem to prove Theorem~\ref{thm:FN} and Corollary~\ref{thm:C0-est} in Sections \ref{subsec:proof_thmFN} and \ref{subsec:cor} and finally show the sharpness of our results by proving Theorem~\ref{thm:sharp}. Many of these proofs are based on energy estimates for the first eigenfunction that we collect in Section~\ref{subsec:energy}, and whose proof is then carried out in the final Section~\ref{sect:proof_ef}.
\section{Background material}\label{sect:background}

\subsection{Hyperbolic surfaces and collars}$ $\\ 
In this section we collect results on hyperbolic surfaces and collars that we will use in the main parts of this paper. These results are all well-known and can be found e.g. in the books of Buser \cite{Buser} and of Hummel \cite{Hu} on hyperbolic surfaces.

We will repeatedly use that a neighbourhood of any simple closed geodesic is described by the following Collar lemma of Keen-Randol

\begin{lemma}[Keen-Randol \cite{randol}] \label{lemma:collar}
Let $(M,g)$ be a closed oriented hyperbolic surface and let $\si$ be a simple closed geodesic of length $\ell$. Then there is a neighbourhood $\Col(\si)$ around $\si$, a so-called collar, which is isometric to the 
cylinder 
$(-X(\ell),X(\ell))\times S^1$
equipped with the metric $\rho^2(s)(ds^2+d\theta^2)$ where 
\beq \label{eq:rho-X} 
\rho(s)=\rho_\ell(s)=\frac{\ell}{2\pi \cos(\frac{\ell s}{2\pi})} 
\qquad\text{ and }\qquad  
X(\ell)=\frac{2\pi}{\ell}\left(\frac\pi2-\arctan\left(\sinh\left(\frac{\ell}{2}\right)\right) \right).\eeq
\end{lemma}

On collars we will always use these coordinates and the corresponding complex variable $z = s + \i\th$. 

It is useful to remark that on collars around geodesics of length $\ell\in (0,2\arsinh(1)]$ we have 
\beq 
\label{est:inj-by-rho}
\rho(s)\leq \inj_g(s,\th)\leq \pi \rho(s) \text{ and } 
\rho(s+\La)\leq \rho(s)e^{\La}
\text{ for all } (s,\th)\in \Col(\si),\, \La>0
\eeq
as a short calculation shows, see e.g. \cite[(A.7) and (A.9)]{RT-neg}. It is also  well-known
that $\de\thin(\Col(\si))\define \{p\in \Col(\si):\ \text{inj}_g(p)<\delta\}$ is given in collar coordinates by 
\beq \label{eq:Xde} (-X_\de(\ell), X_\de(\ell)) \times S^1 \subset \Col(\si),    \text{ where } 
X_\de(\ell)=  \frac{2\pi}{\ell}\left(\frac{\pi}{2}-\arcsin \left(\frac{\sinh(\frac{\ell}{2})}{\sinh \de}\right) \right)\eeq
for $\de\geq \ell/2$, respectively $X_\de(\ell)=0$  for smaller values of $\de$ 
and that 
\beq
\Area(\de\thin(\Col(\si))\leq C\de .
\label{est:area-thin}
\eeq

We furthermore recall from \cite[Theorem~4.1.1]{Buser} that any set $\{\si^1, \ldots, \si^k\}$ of simple closed disjoint geodesics in $(M,g)$
can be extended to a decomposing collection $\mathcal{E}=\{\si^1, \ldots, \si^{3(\gamma-1)}\}$ of simple closed geodesics which can and will always be chosen so that the following holds:

\begin{lemma}(Consequence of \cite[Theorem~3.7]{Hu}) \label{lemma:appendix-collect}
For any genus $\gamma\geq 2$ and any number $\bar L_1$ there 
 exists a number $\bar L$ so that the following holds true: 
 Let $\{\si^1,\ldots, \si^k\}$ be any set  of disjoint simple closed geodesics 
in a hyperbolic surface $(M,g)$ of genus $\gamma$ whose lengths are $L_g(\si^j)\leq \bar L_1$, $j=1,\ldots, k$. Then this set 
can be extended to a collection $\mathcal{E}=\{ \si^1, \ldots, \si^{3(\gamma -1)}\}$ of disjoint simple closed geodesics that decomposes $(M,g)$ into pairs of pants and that is chosen
so that 
$L_g(\si^i)\leq \bar L
$ for each $i$.
\end{lemma}

\subsection{Standard properties of holomorphic quadratic differentials}\label{sec:hol}$ $\\
Throughout the paper we make use of well-known properties of holomorphic quadratic differentials on hyperbolic surfaces, which we summarise in the present section. None of these properties are new and the stated estimates are in particular already contained in the work of Wolpert \cite{Wolpert82, WII, Wolpert12}. 
The present section is included for the convenience of the reader and as our notation is quite different from the one in \cite{Wolpert82, WII, Wolpert12}, and we note that these estimates can also be found in \cite{RT-neg}. 

To begin with we recall from  \cite{tromba} that the tangent space to $\M$ splits $L^2$-orthogonally as 
$T_g\M= \{L_Xg, X\in \Gamma(TM)\}\oplus \Rea(\Hol(M,g))$
for
$$\Hol(M,g):= \{\Psi: \text{ holomorphic quadratic differentials on } (M,g)\}$$
We recall that real parts of holomorphic quadratic differentials are given by trace and divergence free elements of $T^*M\otimes T^*M$. Hence we will always compute the inner products of such real tensors using the standard inner product on $T^*M\otimes T^*M$ induced by $g$, so that e.g. $\abs{dx\otimes dx}_g^2=\rho^{-2}$ if   $g=\rho^2(dx^2+dy^2)$.
For quadratic differentials we use the normalisation 
$\abs{dz^2}_g=2\rho^{-2}$ for $g=\rho^2(dx^2+dy^2)$ so that inner products of quadratic differentials are given by 
$\langle \psi dz^2,\phi dz^2\rangle= 2\psi \bar\phi \rho^{-2}$. We will use in particular that with this normalisation 
\beq 
\label{eq:Re_inner_prod} 
\langle \Rea(\Psi), \Rea(\Phi)\rangle_{L^2(M,g)} =\thalf \Rea\langle \Psi, \Phi\rangle_{L^2(M,g)}, \text{ in particular } \norm{\Psi}_{L^2}^2=2\norm{\Rea(\Psi)}_{L^2}^2. \eeq

We will furthermore use that norms over the thick part of the surface \[\de\thick(M,g):=\{p\in M: \inj_g(p)\geq \de\}\] are controlled by 
\beq 
\label{est:Linfty-by-L1-1} 
\norm{\Upsilon}_{L^\infty(\de\thick(M,g))}\leq C_\de \norm{\Upsilon}_{L^1(\tfrac\de2\thick(M,g))} \text{ for any } \Upsilon\in \Hol(M,g)
\eeq
for a constant $C_\de$ that depends only on $\de>0$. 

We also recall that for $0<\de<\arsinh(1)$ the $\de\thin$ part of the surface is contained in the union of the collars $\Col(\si)$ around simple closed geodesics $\si$ of length less than $2\de$. 

On such a collar $\Col(\si)$ around a simple closed geodesic $\si$ we will always use collar coordinates $(s,\th)\in (-X(\ell),X(\ell))\times S^1$ as described in the Collar lemma~\ref{lemma:collar}
and always set $z=s+\i\th$. We will often use that on $\Col(\si)$ we
may represent any $\Upsilon\in \Hol(M,g)$ by 
its Fourier series
\beq \label{eq:Laurent}
\Upsilon= \sum_{n=-\infty}^\infty b_n(\Upsilon) e^{n(s+\i\theta)} dz^2,\qquad\quad b_n(\Upsilon)=b_n(\Upsilon, \Col(\si)) \in \C, 
z=s+\i\th \eeq
and that on $\Col(\si)$ we may split $\Upsilon$ orthogonally into its principal part $b_0(\Upsilon)dz^2$ and its collar decay part $\Upsilon-b_0(\Upsilon)dz^2$. 
In situations where we are dealing with a fixed collection $\{\si^j\}$ of geodesics we will also use the abbreviation $b_0^j(\Upsilon):=b_0(\Upsilon,\Col(\si^j))$. 

We will also use that for collars $\Col(\si)$ around geodesics of length $L_g(\si)\leq 2\arsinh(1)$
\beq\label{eq:dz-rho}
\abs{dz}_g^2=2\rho^{-2}\ \text{and}\ \norm{dz^2}_{L^1(\Col(\si))}=8\pi X(\ell) \text{ and } \|dz^2\|_{L^2(\Col(\si))}^2= 32\pi^5 \ell^{-3}+O(1),\eeq
as a short calculation shows, as well as that
\beq
\label{est:dz2-est}
 \norm{dz^2}_{L^\infty(\de\thick(\Col(\si)))}\leq C\de^{-2}.
\eeq

We note that since $\norm{dz^2}_{L^2}\geq c(\bar L)\ell^{-3/2}>0$,  the coefficient describing the principal part can always be bounded by  
\beq 
\label{est:b0-trivial-small} 
\abs{b_0(\Upsilon,\Col(\si))}\leq C \ell^{3/2} \norm{\Upsilon}_{L^2(M,g)} \text{ for } C=C(\bar L), \quad \bar L \text{ an upper bound on } L_g(\si),
\eeq 
and in case $L_g(\si)\leq \arsinh(1)$ furthermore by
\beq 
\label{est:b0-upper} 
\abs{b_0(\Upsilon,\Col(\si))}\leq C  \norm{\Upsilon}_{L^2(\half \arsinh(1)\thick(\Col(\si)))}\ell^{3/2}.
\eeq 

Conversely, to bound the collar decay part we use that for any $0<\de<\half \arsinh(1)$
\begin{equation}\label{est:W1} \norm{\Upsilon-b_0(\Ups)dz^2}_{L^\infty(\de\thin(\Col(\si)))}\leq  C\de^{-2}e^{-\pi/\de}\norm{\Upsilon}_{L^2(\half \arsinh(1)\thick(\Col(\si)))}.
\end{equation}

\subsection{Dual bases to differentials of length and twist coordinates}\label{sec:dual}$ $\\
To control the dependence of the first eigenvalue on the Fenchel-Nielsen coordinates we will use several different bases of the space of holomorphic quadratic differentials (respectively of a suitable subspace). In the present section we introduce these bases, which were studied in detail in the previous work of the authors \cite{holo-paper} respectively of Topping and the second author \cite{RT-neg}, recall the relevant results from \cite{holo-paper} and \cite{RT-neg} on which the later analysis is built and  explain how these bases can be used to compute derivatives of functions such as eigenvalues with respect to the Fenchel-Nielsen coordinates.
We remark that while similar estimates for related bases of $\Hol(M,g)$ were already obtained by Wolpert in \cite{Wolpert08}, compare also \cite{Masur, Mazzeo-Swoboda, Yamada1}, for us it is important to work with bases that are dual to the corresponding differentials, as considered in \cite{RT-neg} and \cite{holo-paper} (rather then e.g. bases obtained as gradients of $\ell_j$ as considered in \cite{Wolpert82} and \cite{Wolpert08}).

To begin with we  
recall the well-known evolution equation for the length of simple closed geodesics along horizontal curves of hyperbolic metrics, which is present already in the work of  Wolpert \cite{Wolpert82}: 
Let $g(t)$ be a family of hyperbolic metrics that moves in the direction $\partial_t g=\Rea(\Upsilon(t))$ of holomorphic quadratic differentials $\Upsilon(t) \in \Hol(M,g(t))$ and let $\si^j$ be a given simple closed geodesic in $(M,g(0))$. Then the length $\ell_j(t)= L_{g(t)}(\si^j(t))$ of the unique geodesic $\si^j(t)$ in $(M,g(t))$ homotopic to $\si$ evolves according to 
\beq 
\label{eq:length-evol}
\tfrac{d}{dt}\, \ell_j=-\tfrac{2\pi^2}{\ell_j} \Rea(b_0^j(\Upsilon)), \text{ for } b_0^j(\Upsilon):=b_0(\Upsilon, \Col(\si^j)),
\eeq
where as above 
$b_0(\Upsilon,\Col(\si))dz^2$ denotes the principal part of $\Upsilon$ on a collar $\Col(\si)$ around a simple closed geodesic. 
We
can hence
 consider the 
$\C$-linear differentials of the length coordinates $\ell_j$ on $\Hol(M,g)$ described by 
\beq
\label{def:dell}
\partial \ell_j\colon\Hol(M,g)\to \C, \qquad \partial \ell_j(\Upsilon):=
-\tfrac{\pi^2}{\ell_j}b_0^j(\Upsilon), \eeq
 compare  \cite[Remark 4.1]{RT-neg} and \cite{Wolpert82}.

For the analysis of eigenvalues on surfaces for which some  geodesics, say $\si^1,\ldots, \si^k$, collapse, it  turns out to be useful to follow the approach of Topping and the second author from \cite{RT-neg}: we split $\Hol(M,g)$ into $\ker(\partial \ell_1,\ldots, \partial \ell_k)$ and its orthogonal complement and  consider the dual basis $\{\tilde \Th^j\}_{j=1}^k$ of the map $(\partial \ell_1,\ldots, \partial \ell_k)\colon\ker(\ell_1,\ldots, \partial \ell_k)^\perp\to \C^k$, which is an isomorphism thanks to \cite[Theorem~3.7]{Wolpert82}.
In the context of our main results, where we only have one collapsing geodesic, we will see that the $L^2$-gradient of 
$\la_1$ can be characterised essentially explicitly in terms of just the corresponding element $\tilde \Th^1$, which is defined by 

\begin{defn}\label{def:Th-tilde}
Let $(M,g)$ be as in Theorem~\ref{thm:FN}. Then we define 
$\tilde{\Th}^1$ be the element of $\ker(\partial \ell_1)^\perp$ which satisfies $\partial \ell_1(\tilde{\Th}^1)=1$ and furthermore set $\tilde{\Om}^1 \define -\frac{\tilde{\Th}^1}{\Vert \tilde{\Th}^1\Vert_{L^2(M,g)}}$.
\end{defn}

The dual basis $ \{\tilde \Th^j\}$ of $\ker(\partial \ell_1,\ldots, \partial \ell_k)$ and its renormalisation $\{\tilde \Om^j\}$ was considered in detail in \cite{RT-neg} and, in our case of only one collapsing geodesic $\si^1$, we know  that $\tilde \Om^1$ has the following properties

\begin{lemma} \label{lemma:RT-neg} (Corollary of \cite[Lemma~4.5]{RT-neg})
Let $(M,g)$ be as in Theorem~\ref{thm:FN}. Then  
$\tilde{\Om}^1 \define -\frac{\tilde{\Th}^1}{\Vert \tilde{\Th}^1\Vert_{L^2}}\in \ker(\partial \ell_1)^\perp$ described in Definition~\ref{def:Th-tilde} satisfies 
\begin{align}\label{est:RT-neg1}
  \Vert \tilde{\Om}^1\Vert_{L^\infty(M\setminus \Col(\si^1), g)} +
\Vert \tilde{\Om}^1 - b_0^1(\tilde{\Om}^1) dz^2\Vert_{L^\infty(\Col(\si^1), g)}\leq C \ell_1^{3/2}
\end{align}
and
\begin{align}\label{est:RT-neg3}
 1-C\ell_1^3\leq b_0^1(\tilde{\Om}^1) \Vert dz^2\Vert_{L^2(\Col(\si^1),g)} \leq 1 
\end{align}
for constants $C$ that depend only on the genus $\gamma$ of $M$ and the lower bound $\hat \de$ on $\inj\vert_{M\setminus \Col(\si^1)}$. 
\end{lemma}

We will use that the above lemma implies in particular that  
\beq 
\label{est:Om-tilde-thick}
 \norm{\tilde\Om^1}_{L^\infty(M,g)}\leq C\ell_1^{-1/2}
 \text{ while }\norm{\tilde \Om^1}_{L^\infty(\frac{\bar\de}{2}\thick (M,g))}\leq C_{\hat\de}\ell_1^{3/2}
 \eeq
 for  $\bar\de=\min(\hat\de,\arsinh(1))$
 and 
\beq \label{est:Om-tilde-L1}
\norm{\tilde \Om^1}_{L^1(M,g)}\leq C\ell_1^{1/2}
\eeq
see also \cite[Lemma~3.12]{Wolpert08} for a closely related result on the corresponding gradient basis. 

We also remark that $\tilde \Om^1$ can be equivalently characterised as the unit element of $\ker(\partial \ell_1)^\perp$ for which $b_0(\tilde \Om^1,\Col(\si^1))>0$ and that 
\beq\label{eq:tilde-Th-b0} \tilde \Th^1 = -\frac{\ell_1}{\pi^2 b_0(\tilde\Om^1, \Col(\si^1))} \tilde\Om^1,\eeq
since \eqref{def:dell} implies that $b_0(\tilde \Om^1,\Col(\si^1))=-\frac{\pi^2} {\ell_1}$. 

We note that as observed in \cite[Corollary 2.3]{holo-paper} the above result from \cite{RT-neg} furthermore allows us to conclude that 
 \beq 
 \label{est:L2-tilde-theta} 
\big| \Vert\tilde \Th^1\Vert_{L^2(M,g)} -\tfrac{\ell_1}{\pi^2} \Vert dz^2\Vert_{L^2(\Col(\si^1),g)}\big|\leq C\ell_1^{5/2} \text{ and }
\big|\norm{\tilde \Th^j}_{L^1(M,g)}-8\pi\big|\leq C\ell_j. 
\eeq

We also recall from  
 \cite{RT-neg} that in the above setting the elements of $\ker(\partial \ell_1)$  are all controlled by
 \beq 
\label{est:W2}
\norm{w}_{L^\infty(M,g)}\leq C_{\hat \de}  \norm{w}_{L^1(M,g)}.
\eeq

Let now  $\mathcal{E}=\{\si^1,\ldots,\si^{3(\gamma-1)}\}$ be a disconnecting family of a
simple closed geodesics 
in a hyperbolic surface $(M,g)$ and let $(\ell_j,\psi_j)$ be the corresponding Fenchel-Nielsen coordinates. Then in addition to the 
$\C$-linear differentials $\partial \ell_j$ of the length coordinates $\ell_j$ on $\Hol(M,g)$ described by 
\eqref{def:dell}, we also need to consider the real differentials of both the length and the twist coordinates on $\Rea(\Hol(M,g))$, which are defined as derivatives 
\beq \label{def:dell-psi-real}
d\ell_j(\Rea(\Upsilon))= \frac{d\ell_j}{dt}=-\tfrac{2\pi^2}{\ell_j} \Rea(b_0^j(\Upsilon)) \text{ and } 
 d\psi_j(\Rea(\Upsilon))= \frac{d\psi_j}{dt},
 \eeq
of the Fenchel-Nielsen coordinates $\ell_j$ and $\psi_j$ along a curve $g(t)$ of hyperbolic metrics with $\partial_t g_t=\Rea (\Upsilon(t))$.

 We note that  \cite[Theorem~3.7]{Wolpert82} assures that $$\Upsilon \mapsto (\partial \ell_1(\Upsilon),\ldots,\partial \ell_{3(\gamma-1)}(\Upsilon))$$ is an isomorphism from $\Hol(M,g)$ to $\C^{3(\gamma-1)}$, while
 \beq \label{def:secnd-iso} 
 \Rea(\Upsilon)\mapsto 
(d\ell_1(\Rea \Upsilon ),d\psi_1(\Rea\Upsilon),\ldots ,
d\ell_{3(\gamma-1)}(\Rea\Upsilon),d\psi_{3(\gamma-1)}(\Rea \Upsilon))\eeq is an isomorphism from $\Rea(\Hol(M,g))$ to $\R^{6(\gamma-1)}$. 

This allows us to consider the following dual bases which play a key role in our analysis of the dependence of the first eigenvalue on the Fenchel-Nielsen coordinates.

\begin{defn}\label{defn:basis1}
Let $(M,g)$ be a hyperbolic surface. Then we associate to any given disconnecting family $\mathcal{E}=\{\si^1,\ldots,\si^{3(\gamma-1)}\}$ of 
simple closed geodesics the following dual bases. 
\begin{enumerate}[(i)]
\item We let $\{\Th^j\}_{j=1}^{3(\gamma-1)}$ be the basis of $\Hol (M, g)$ which is
 dual to complex differentials of the length coordinates $\{\partial \ell_j\}_{j=1}^{3(\gamma-1)}$ given by \eqref{def:dell}, i.e. characterised by
 $$\partial \ell_j(\Th^i)=\delta_j^i \text{ for } i,j\in \{1,\ldots, 3(\gamma-1)\}.$$
 We furthermore denote by $\{\Om^j\}$ the renormalised dual  basis whose elements are given by 
\beq 
\label{def:Om-by-Th} \Om^j\define -\frac{\Th^j}{\Vert \Th^j\Vert_{L^2(M,g)}}.
\eeq  
 \item We let $\{\Psi^j,\La^j\}_{j=1}^{3(\gamma-1)}$ be the elements of $\Hol(M,g)$ which are dual to the real differentials of the Fenchel-Nielsen coordinates in the sense that for  $i,j\in \{1,\ldots, 3(\gamma-1)\}$
  \beq \label{def:Psi_La}
 d\ell_j(\Rea(\La^i))=\de_j^i=d\psi_j(\Rea(\Psi^i)) \text{ and } 
 d\psi_j(\Rea(\La^i))=0=d\ell_j(\Rea(\Psi^i)).
 \eeq
 That is $\{\Psi^j,\La^j\}$ are the unique elements of $\Hol(M,g)$ for which $\{\Rea(\Psi^j), \Rea(\La^j)\}$ is the dual basis of $\Rea(\Hol(M,g))$ for the isomorphism \eqref{def:secnd-iso}.
\end{enumerate}
\end{defn}

We  remark that for 
any function $f\colon \M\to \R$ which is invariant under pull-back by diffeomorphisms we can express the derivatives of $f$ with respect to a given set of Fenchel-Nielsen coordinates $(\ell_j,\psi_j)$ in terms of the dual basis $\{\Psi^j,\La^j\}$ and the $L^2$-gradient of $f$, namely  
\beq 
\label{eq:deriv-f}\frac{\partial f}{\partial \ell_j}=\langle \na f, \Rea(\La^j)\rangle, \qquad 
\frac{\partial f}{\partial \psi_j}=\langle \na f, \Rea(\Psi^j)\rangle. 
\eeq

In the main part of the paper we will use this idea to
obtain sharp bounds on the derivatives of the eigenvalue $\la_1$ from an essentially explicit expression for $\na \la_1$ (in terms of $\tilde \Th^1$) that we will obtain in Theorem~\ref{thm:1} and 
precise bounds on the above dual bases and their relations. 
Such bounds were derived in \cite{holo-paper}, compare also \cite{RT-neg}, and we recall the relevant estimates here. These estimates will all be valid for constants 
that only depend on the genus and on numbers  $\eta,\bar L>0$ which are so that 
\begin{align}
 \mathcal{E}\ \textrm{contains all simple closed geodesics\ } \si \textrm{ of } (M,g) 
 \textrm{ of length } L_g(\si)\leq 2\eta\label{ass:eta}
\end{align}
and
\begin{align}\label{ass:upperbound}
\ L_g(\si)\leq \bar L \textrm{ for every } \si\in \mathcal{E}.
\end{align} 

For the dual bases $\{\Th^j\}$ and $\{\Om^j\}$ of the complex differentials $\partial \ell_j$ we will use the following result from \cite{holo-paper}, which gives the same type of estimates as obtained in the result from \cite{RT-neg} for the $\tilde \Om^j$ that we recalled above. 
 \begin{prop}\cite[Prop. 1.1 and Lem. 2.9]{holo-paper} \label{prop:RT-new} Let $(M,g)$ be any closed oriented hyperbolic surface of genus $\gamma\geq 2$. Let $\mathcal{E}=\{\si^1, \ldots, \si^{3(\gamma-1)}\}$ be any set of simple closed geodesics that decompose $(M,g)$ into pairs of pants.
Then 
 $\{\Om^j\}_{j=1}^{3(\gamma-1)}$ respectively $\{\Th^j\}_{j=1}^{3(\gamma-1)}$ from Definition~\ref{defn:basis1} satisfy
\begin{align}
\label{est:RT-neg1-new}
 \Vert  {\Om}^j\Vert_{L^\infty(M\setminus \Col( {\si}^j), g)} + 
 \Vert  {\Om}^j - b_0^j( {\Om}^j)dz^2\Vert_{L^\infty(\Col (\sigma^j), g)} &\leq C \ell_j^{3/2},\\
  \label{est:RT-neg3-new}
\max(1-C\ell_j^3,\eps_1)\leq b_0^j( {\Om}^j) \Vert dz^2\Vert_{L^2(\Col( {\si}^j),g)} &\leq \, 1,\\
  \label{est:RT-neg4-new}
 |\langle  {\Om}^i,  {\Om}^j\rangle_{L^2(M,g)}|&\leq \, C\ell_i^{3/2} \ell_j^{3/2} \text{ for every } i\neq j
\end{align} 
and
  \beq 
 \label{est:Th-L2-upper} 
 \norm{\Th^j}_{L^2(M,g)}\leq C\ell_j^{-1/2}
 \eeq
 for every $j=1,\ldots, 3(\gamma-1)$ and 
for constants  $C\in\R$ and $\eps_1>0$ that depend only on the genus and the numbers 
 $\eta\in (0,\arsinh(1))$ and $\bar L<\infty$ for which \eqref{ass:eta} and \eqref{ass:upperbound} are satisfied. 
\end{prop}

We note that while the elements $\tfrac12\Th^i$ and $\La^i$ induce the same change of the length coordinates, namely 
$d\ell_j(\tfrac12\Rea\Th^i)=\de_{j}^i=d\ell_j(\Rea(\La)) $,
the elements $\Th^j$ of the dual basis of the complex differentials $\partial \ell_j$ will in general not leave the twist coordinates invariant so we cannot expect these two elements to agree. However, the results of \cite{holo-paper} assure that the difference between these elements is only of order 
 $O(\ell_j)$ and furthermore allow us to express the $\La^j$ and $\Psi^j$ in terms of the $\Om^j$ as follows.

\begin{prop}\cite[Thm. 1.2 and 1.3]{holo-paper} \label{prop:La-Psi}
 Let $(M,g)$ be any closed oriented hyperbolic surface of genus $\gamma$, let $\mathcal{E}=\{\si^1, \ldots, \si^{3(\gamma-1)}\}$ be any decomposing set of simple closed geodesics and let $\{\Th^j \}$, $\{\Om^j\}$ and $\{\Psi^j,\La^j\}$ be the corresponding dual bases defined in Definition~\ref{defn:basis1}.
 
 Then there exists coefficients $c_k^j, d_k^j\in\R$ and $a_j\in\R^+$ so that 
\beq
\label{eq:writing_La} 
\La^j=\tfrac12\Th^j+\sum_k \i\, c_k^j \Om^k\ \text{
 with }\abs{c_k^j}\leq C\ell_j\ell_k^{3/2},\quad j,k=1,\ldots, 3(\gamma-1) 
\eeq
and
\beq
\label{eq:write-Psi-with-Om}
\frac{\Psi^j}{\Vert \Psi^j\Vert_{L^2(M,g)}}= -a_j \i\Om^j+\i\sum_{k\neq j} d_k^j \Om^k \text{ with }  \abs{d_k^j}\leq C\ell_k^{3/2} \ell_j^{3/2} \text{ and } \abs{1-a_j}\leq C\ell_j^3
\eeq
where  $C$ depends only on the genus and the numbers $\eta\in (0,\arsinh(1))$ and $\bar L<\infty$ for which \eqref{ass:eta} and \eqref{ass:upperbound} are satisfied. 

In particular 
\beq \label{est:La-minus-Th}
\norm{\La^j-\tfrac12\Th^j}_{L^\infty(M,g)}\leq C\ell_j \text{ and } \norm{\La^j}_{L^\infty(M,g)}\leq C\ell_j^{-1}.
\eeq 
Furthermore, 
\beq 
\label{est:L2-Psi-upper}
\norm{\Psi^j}_{L^2(M,g)}\leq C\ell_j^{3/2}
\eeq
and $\Psi^j$ is  orthogonal to $\ker(\partial \ell_j)=\text{span}\{\Om^i\}_{i\neq j}=\text{span}\{\Th^i\}_{i\neq j}$.
\end{prop}
We note that the orthogonality of $\Psi^j$ to $\ker(\partial \ell_j)$ is already a consequence of Wolpert's length-twist duality \cite[Theorem~2.10]{Wolpert82}.

The above estimates imply in particular that the principal parts of $\La^j$ satisfy 
\beq 
\label{est:princ-parts-La}
\Rea(b_0^k(\La^j))=-\frac{\ell_j}{2\pi^2}\de_{kj} \text{ and }
\abs{\Ima(b_0^k(\La^j))}=\abs{c_k^j}\abs{b_0^k(\Om^k)} \leq C\ell_j\ell_k^{3}.
\eeq
Combining 
Proposition \ref{prop:La-Psi}  with Proposition~\ref{prop:RT-new} furthermore yields that 
 for any number $\de_0\in (0,\eta)$ there is a constant $C=C(\de_0,\gamma)$ so that, after relabelling the geodesics $\si^j$ to assure that $\ell_j\leq 2\de_0$ precisely for $j\leq j_0\in\{0,\ldots, 3(\gamma-1)\}$, we can bound 
\beq\label{est:La-for-sharp}
\sum_{k=1}^{j_0} \Vert \La^j - b_0(\La^j, \Col(\si^k)) dz^2\Vert_{L^\infty(\Col(\si^k), g)}+
 \Vert \La^j\Vert_{L^\infty(\de_0\thick(M, g))}\leq C \ell_j.
\eeq
As we shall be able to characterise $\na \la_1$ in terms of $\tilde \Th^1$, it is also useful to recall that there is the 
following close relationship between 
 $\Th^1$ and $\tilde \Th^1$, respectively between the renormalised elements $\Om^1$ and  $\tilde \Om^1$:

\begin{prop}\label{prop:Th-tildeTh}\cite[Lemma 2.9]{holo-paper}
Let $(M,g)$ and $\mathcal{E}$ be as in Theorem~\ref{thm:FN}, 
let $\{\Th^j\}_{j=1}^{3(\gamma-1)}$ be the dual basis of $\Hol(M,g)$ introduced in Definition~\ref{defn:basis1} and let $\tilde \Th^1\in \ker(\partial \ell_1)^\perp$ 
be as in Definition~\ref{def:Th-tilde}. Then 
\beq
\label{eq:Th-by-tilde-Th} 
\Th^1= \tilde\Th^1+v^1 \text{ for some } 
 v^1\in \ker(\partial \ell_1) \text{ with }  \norm{v^1}_{L^\infty(M,g)}\leq C \ell_1
\eeq
while the renormalised elements $\Om^1$ and $\tilde \Om^1$ 
are related by 
\beq
\label{eq:Om-by-tilde-Om}
\Om^1= \beta_1 \tilde \Om^1 +w^1 \text{ for some } 
 w^1\in \ker(\partial \ell_1) \text{ with }  \norm{w^1}_{L^\infty(M,g)}\leq C \ell_1^{3/2}
\eeq
and some $\beta_1\in\R^+$ with $1+C\ell_1^{3/2}\geq \beta_1\geq \max(1-C\ell_1^{3/2}, \eps_2)$ for constants $C$ and $\eps_2>0$ that depend only 
on the genus $\gamma$ of $M$ and the lower bound $\hat \de$ on $\inj\vert_{M\setminus \Col(\si^1)}$. 
\end{prop}

Combined with \eqref{est:L2-tilde-theta} the above lemma implies in particular that
\beq \label{est:L^1-the1-precise} 
|\norm{\Th^1 }_{L^2(M,g)}^2- \tfrac{32\pi}{\ell_1}| \leq C \ell_1^2.\eeq

As the derivatives of $\la_1$ with respect to the Fenchel-Nielsen coordinates are given by inner products of $\na\la_1$ with the elements $\{\Psi^j, \La^j\}$, it is furthermore useful to prove the following estimates on the inner products of elements of the bases of $\Hol(M,g)$ introduced in Definition~\ref{defn:basis1}.

\begin{lemma}\label{lemma:inner-prod}
Let $(M,g)$ be any closed oriented hyperbolic surface of genus $\gamma$. Let $\mathcal{E}=\{\si^1, \ldots, \si^{3(\gamma-1)}\}$ be any decomposing set of simple closed geodesics and let $\{\Th^j \}$, $\{\Om^j\}$ and $\{\Psi^j,\La^j\}$ be the dual bases defined in Definition~\ref{defn:basis1}.
Then their inner products are bounded by the following estimates which hold true for constants $C$ that depend only on the genus and the numbers $\eta$ and $\bar L$ from \eqref{ass:eta} and \eqref{ass:upperbound}.
For any $j\in\{1,\ldots, 3(\gamma-1)\}$ we have 
\begin{align}
\label{est:inner-Th-La-same-Re}
\abs{\langle \Rea(\Th^j),\Rea(\La^j)\rangle-\tfrac14 \norm{\Th^j}_{L^2}^2}&\leq C\ell_j^2,\\
\label{est:inner-Th-La-same-Im}
\abs{\langle\Rea(\La^j),\Rea(\i \Om^j)\rangle}&\leq C\ell_j^{5/2},
\\
\label{est:inner-Om-Psi-same}
\abs{\langle\Rea(\tfrac{\Psi^j}{\norm{\Psi^j}_{L^2}}),\Rea( \Om^j)\rangle}&\leq C\ell_j^{3}, 
\end{align}
while for any $i\neq j$ 
\begin{align}
\label{est:inner-Om-La-diff}
\abs{\langle\La^j, \Om^i\rangle}&\leq C\ell_j\ell_i^{3/2}.
\end{align}
\end{lemma}
As mentioned above, Wolpert's length-twist duality already implies that $\langle\Psi^j,\Om^i\rangle=0$ for $i\neq j$. 
\begin{proof}
The proofs of all these estimates are obtained by combining the expressions for $\La^j$ respectively $\Psi^j$ given in Proposition~\ref{prop:La-Psi} with the estimates 
$\abs{\lan \Om^j, \Om^i\ran}\leq \de_{ji}+C\ell_j^{3/2}\ell_i^{3/2}$ and $\norm{\Th^j}_{L^2}\leq C\ell_j^{-1/2}$  from Proposition~\ref{prop:RT-new} . 

We begin by proving the estimates on the inner products involving $\La^j$ which, by Proposition~\ref{prop:La-Psi}, is given by $\La^j=\frac12\Th^j+\sum_j\i\cdot c_k^j\Om^k$ for $c_k^j\in\R$ satisfying  $|c_k^j|\leq C\ell_j\ell_k^{3/2}$. 
To prove \eqref{est:inner-Th-La-same-Re} we combine the above expression with  \eqref{eq:Re_inner_prod} 
to write 
$$\abs{\lan \Rea\Th^j,\Rea\La^j\ran-\tfrac14 \norm{\Th^j}_{L^2}^2}
= \abs{\lan \Rea\Th^j,\Rea(\La^j-\tfrac12
\Th^j)\ran}= \abs{\lan \Rea\Th^j,\Rea(\sum_k\i\cdot c_k^j \Om^k)\ran}
.$$  
As $\Rea\Th^j=-\norm{\Th^j}_{L^2}\Rea(\Om^j)\perp \Rea(\i\Om^j)$ 
we thus obtain the claimed bound of
\beqas
\abs{\lan \Rea\Th^j,\Rea\La^j\ran-\tfrac14 \norm{\Th^j}_{L^2}^2}&\leq \norm{\Th^j}_{L^2}\sum_{k\neq j}\abs{c_k^j}\abs{\langle \Om^j,\Om^k\rangle}\leq C\ell_j^{-1/2}\sum_{k\neq j} \ell_j\ell_k^{3/2} \cdot \ell_j^{3/2}\ell_k^{3/2}\leq C\ell_j^2.
\eeqas
Similarly, we obtain the second claim of the lemma by estimating
\beqas
\abs{\lan \Rea(\La^j),\Rea(\i\Om^j)}=\abs{0+\sum_{k} c_{k}^j \lan \Rea(\i \Om^k),\Rea(\i \Om^j)\ran} \leq C\sum_{k} \ell_j\ell_k^{3/2} (\de_{jk} +C\ell_j^{3/2}\ell_k^{3/2})\leq C\ell_j^{5/2}.
\eeqas 
Finally, to obtain \eqref{est:inner-Om-La-diff}, we use that for $i\neq j$ 
\beqas
\abs{\lan \La^j,\Om^i\ran}&\leq \thalf \norm{\Th^j}_{L^2} \abs{\lan\Om^j,\Om^i\ran}+\sum_{k}\abs{c_k^j}\abs{\lan\Om^k,\Om^i\ran} \\
&
\leq C\ell_j^{-1/2} \ell_j^{3/2}\ell_i^{3/2}+C\sum_{k} \ell_j\ell_k^{3/2} (\de_{ki}+C\ell_k^{3/2}\ell_i^{3/2})\leq C\ell_j\ell_i^{3/2}.
\eeqas
To prove \eqref{est:inner-Om-Psi-same} we use Proposition~\ref{prop:La-Psi} to write 
 $\frac{\Psi^j}{\Vert \Psi^j\Vert_{L^2(M,g)}}= -a_j \i\Om^j+\i\sum_{k\neq j} d_k^j \Om^k $ 
for real coefficients $a_j>0$ and $|d_k^j|\leq C\ell_k^{3/2}\ell_j^{3/2}$ to conclude that indeed
\beqas
\abs{\langle\Rea(\frac{\Psi^j}{\norm{\Psi^j}_{L^2}}),\Rea( \Om^j)\rangle}=\abs{0+\sum_{k\neq j}d_k^j\lan\Rea(\i \Om^k),\Rea(\Om^j)\ran}\leq C\sum_{k\neq j} C\ell_k^{3/2}\ell_j^{3/2}\cdot C\ell_k^{3/2}\ell_j^{3/2} \leq C\ell_j^3
\eeqas
for any $j\in\{1,\ldots, 3(\gamma-1)\}$ 
as claimed.
\end{proof}

\section{Proofs of the main results} \label{sect:main}
We now turn to the proofs of our main results on the behaviour of the first eigenvalue. 

In the first part of this section we collect properties of the first eigenfunction,
proved later on in Section~\ref{sect:proof_ef}, which we then use in the subsequent section to give an essentially explicit characterisation of the $L^2$-gradient of $\la_1$ in terms of the dual 
$\tilde \Th^1$ of the derivative of the degenerating length coordinate. This characterisation is stated in Theorem~\ref{thm:1}, will be proven in Section~\ref{subsec:char-la} and will at the same time be the basis on which we shall prove all other main results in 
the subsequent sections: we prove Theorem~\ref{thm:FN} in Section~\ref{subsec:proof_thmFN}, Corollary~\ref{thm:C0-est} in Section~\ref{subsec:cor} and finally Theorem~\ref{thm:sharp} in Section~\ref{subsec:proof_sharp}.

\subsection{Properties of the first eigenfunction}
\label{subsec:energy} $ $

We recall that the first eigenvalue $\la_1$ and eigenfunction $u_1$ 
minimise the 
Rayleigh-quotient $\frac {\norm{d v}_{L^2(M,g)}^2}{\norm{v}_{L^2(M,g)}^2}$ over the set of all functions $v\in H^1(M,g) \text{ for which } \int_Mvdv_g=0$
and will use that $u_1$ satisfies the following energy estimates which are proven in Section~\ref{sect:proof_ef}.

\begin{lemma}\label{lemma:est-u-thick-main}
For any $\gamma\geq 2$ there exists a
constant $C_0$ so that the following holds true for any closed oriented hyperbolic surface 
$(M,g)$ of genus $\gamma$ and any number 
$\bar \de\in (0,\arsinh(1)]$.
\newline
Suppose that 
all simple closed 
geodesics $\si^1,\ldots,\si^k$ of length less then $2\bar\de$ are so that $M\setminus \si^j$ is disconnected.
Then the first eigenfunction $u_1$ of $-\Delta_g$ (as always normalised by $\norm{u_1}_{L^2(M,g)}=1$)
satisfies the estimate 
\beq \label{est:u-de-thick}
\norm{du_1}_{L^2(\de\thick(M,g))}^2\leq \frac{C_0}{\de}\lambda_1^2 \text{ for every } 0< \de \leq  \bar \de.
\eeq

\end{lemma}

We shall furthermore need the following estimates on the \textit{angular} energy which hold true for general eigenfunctions of $-\Delta_g$.

\begin{lemma}\label{lemma:ang-energy-mainpart}
There exist universal constants $C_{1,2}$ and $\de_3>0$ so that the following holds true  for any closed oriented hyperbolic surface 
$(M,g)$ and any eigenfunction $u$ of $-\Delta_g$ to an eigenvalue $\la\in \R$ with $\norm{u}_{L^2(M,g)}=1$. 
Let $\si$ be a simple closed geodesic of length $\ell<2\arsinh(1)$, let $\Col(\si)$ be the collar around $\si$ described by the Collar 
Lemma~\ref{lemma:collar} and let $(s,\th)\in (-X(\ell),X(\ell))\times S^1$ be the corresponding collar coordinates in which the metric takes the form $g=\rho^2(ds^2+d\th^2)$, for $\rho$ and $X(\ell)$ as in \eqref{eq:rho-X}.
 Then
\beq \label{est:weighted-ang-en-4}
\int_{-X(\ell)}^{X(\ell)}\int_{S^1}\abs{u_\th}^2\rho^{-4} dsd\theta\leq C_1 \norm{du}_{L^2(\de_3\thick(\Col(\si)))}^2+C_1 \la^2  X(\ell)
\eeq
and 
\beq \label{est:weighted-ang-en-2}
\int_{-X(\ell)}^{X(\ell)}\int_{S^1}\abs{u_\th}^2\rho^{-2} dsd\theta\leq C_2\norm{du}_{L^2(\de_3\thick(\Col(\si)))}^2 +C_2 \la^2\norm{u}_{L^\infty(M,g)}^2 .
\eeq
\end{lemma}

To apply the above lemma for the first eigenfunction $u_1$ we furthermore recall the following well-known fact about the first eigenfunction: 

\begin{rmk} \label{rmk:L-infty} 
There exists a constant $C_3$ depending at most on the genus of $M$ so that the following holds true: Let $(M,g)$ be a  
closed hyperbolic surface whose 
shortest simple closed geodesic $\si$ is such that $M\setminus \si$ is disconnected. Then
the (normalised) first eigenfunction $u$ of $-\Delta_g$
is bounded by $\norm{u}_{L^\infty(M,g)} \leq C_3$.
\end{rmk}

Together with the above  Lemmas~\ref{lemma:est-u-thick-main} and \ref{lemma:ang-energy-mainpart} this last remark directly implies that the angular energy of the first eigenfunction is controlled by

\begin{cor}\label{cor:ang-energy-mainpart}
Let $(M,g)$ be a hyperbolic surface for which the assumptions of 
 Lemma~\ref{lemma:est-u-thick-main} are satisfied for some number $\bar\de>0$ and let $u_1$ be the first eigenfunction of $-\Delta_g$ (as always normalised by $\norm{u_1}_{L^2(M,g)}=1$). 
 Then the angular energy of $u_1$ on any collar 
 $\Col(\si)$ around a 
simple closed geodesic of length $\ell<2\arsinh(1)$
 is controlled by  
\beq \label{est:weighted-ang-en-4-new}
\int_{-X(\ell)}^{X(\ell)}\int_{S^1}\abs{\uth}^2\rho^{-4} dsd\theta\leq C \cdot\la_1^2+C_1\ell^{-1} \la_1^2 
\eeq
and
\beq \label{est:weighted-ang-en-2-new}
\int_{-X(\ell)}^{X(\ell)}\int_{S^1}\abs{\uth}^2\rho^{-2} dsd\theta\leq C \cdot\la_1^2,
\eeq
where $C_1$ is the universal constant from Lemma~\ref{lemma:ang-energy-mainpart}, $C$ depends only on the genus of $M$ and the number $\bar\de$ and where $X(\ell)$ and $\rho$ are as in the Collar Lemma~\ref{lemma:collar} .
\end{cor}

\subsection{Characterisation of the gradient of \texorpdfstring{$\la_1$}{la1}}
\label{subsec:char-la} $ $ 

The goal of this section is to prove that in the setting of our main results the gradient of $\la_1$ is essentially determined in terms of the element $\tilde \Th^1\in\Hol(M,g)$ that we introduced in Definition~\ref{def:Th-tilde}. 
Before we turn to this result that is stated in detail in Theorem~\ref{thm:1} below, 
 we first discuss how the $L^2$-gradient of general eigenvalues, considered as functions on the set $\M$ of 
all smooth hyperbolic metrics on $M$, is characterised.

We first remark that 
since the splitting  
$T_g\M= \{L_Xg, X\in \Gamma(TM)\}\oplus \Rea(\Hol(M,g))$ is
$L^2$-orthogonal, the $L^2$-gradient of any differentiable function $f\colon \M\to \R$ which is invariant under the pull-back by diffeomorphisms will be in $\Rea(\Hol(M,g))$. 

We also recall that if we consider the function  $g\mapsto \la_k(g)$ on the set of all metrics (not necessarily hyperbolic), then this function is differentiable at any $g$ for which $\la_k$ is simple. Furthermore, in this setting the corresponding $L^2$-gradient is given by $-\half [ \Rea(\Phi(u_k,g))+\la_k u_k^2\cdot g]$, where
$\Phi(u_k,g)$ is the Hopf-differential of the normalised $k$-th eigenfunction $u_k$,  given in local isothermal coordinates $(x,y)$ of $(M,g)$ as 
\beq 
\label{def:Hopf-diff}
\Phi(u_k,g)=(\abs{\partial _x u_k}^2-\abs{\partial_y u_k}^2-2\i \langle\partial_x u_k, \partial_y u_k \rangle) dz^2, \qquad z=x+\i y,
\eeq
see e.g. \cite[Lemma 2.2]{Fraser-Schoen-2013}. As an immediate consequence we obtain that if we consider $\la_k$ only as a function on $\M$, then its $L^2$-gradient is given by the $L^2$-orthogonal projection of 
  $-\half [ \Rea(\Phi(u_k,g))+\la_k u_k^2\cdot g]$ onto $\Rea(\Hol(M,g))$ and so, as tensors in $\Rea(\Hol(M,g))$ are trace-free, by: 

\begin{lemma}
\label{cor:grad-la}
Let $(M,g)$ be a hyperbolic surface for which the $k$-th eigenvalue $\lambda_k$ is simple, $k$ any element of $\N$. Let $u_k$ be the corresponding eigenfunction, normalised to have $\norm{u_k}_{L^2(M,g)}=1$.
Then the $L^2
$-gradient of $\lambda_k\colon\M\to \R$ is given by
\beq
\label{eq:grad-la-general} 
\na \lambda_k(g)=-\tfrac12\Rea(\Proj(\Phi(u_k,g)))\eeq
for $\Phi(u_k,g)$ the Hopf-differential given by \eqref{def:Hopf-diff} and $\Proj$ the $L^2$-orthogonal projection from the space of $L^2$-quadratic differentials onto $\Hol(M,g)$.
\end{lemma}
We recall that Remark~\ref{rmk:la-simple} ensures that the first eigenvalue is simple in the situations considered in our main results, allowing us to apply this formula for $\la_1$ and the corresponding (normalised) eigenfunction $u_1$. 

The goal of the present section is to prove the following result, which assures that in the setting of our main results 
 the  $L^2$-gradient of 
 the first eigenvalue $\la_1\colon\M\to\R$ 
 of $-\Delta_g$ is essentially determined by
$$\na \log(\la_1)\sim \tfrac{1}{8\pi} \Rea (\tilde \Th^1)$$ 
for $\tilde\Th^1$ as in Definition~\ref{def:Th-tilde}.

\begin{thm}\label{thm:1}
Let $(M,g)$ be a closed oriented hyperbolic surface, let $\si^1$ be a disconnecting simple closed geodesic.
\newline
Suppose that $\ell_1=L_g(\si^1)\leq \ell_0$ for $\ell_0=\ell_0(\hat \de,\gamma)$ as in Remark~\ref{rmk:la-simple} and $\hat \de>0$ as usual a lower bound on $\inj_g$ on $M\setminus \Col(\si^1)$.   
Then there exists a number $\alpha\geq 0$ with 
\beq 
\label{claim:alpha} 
\abs{\alpha-\tfrac{1}{8\pi}}\leq C\ell_1 \abs{\log(\ell_1)},\eeq
$C$ depending only on the genus of $M$ and on $\hat \de$, such that
\beq
\label{claim:na-la-2}
\norm{\na\log(\la_1)-\al \Rea (\tilde \Th^1)}_{L^\infty(M,g)}\leq C \ell_1
\eeq
for $\tilde \Th^1\in \Hol(M,g)$ characterised by Definition~\ref{def:Th-tilde}.
\end{thm}

The proof of this result, which seems to be of independent interest and will be the basis of the proofs of our other main results, will be carried out in the remainder of this section and is structured as follows: 
We first argue that it suffices to prove the result in case that $\ell_1\leq \hat \eps$ for a number $\hat\eps=\hat\eps(\gamma,\hat\de)\in (0, 2\min(\hat\de,\arsinh(1)))$ chosen later, as it is otherwise trivially true for  $\al=0$ and a suitably chosen constant $C=C(\gamma,\hat\de)$. We will then use the energy estimates on the first eigenfunction $u_1$ obtained in Section~\ref{subsec:energy} to derive bounds on inner products of the Hopf-differential $\Phi=\Phi(u_1,g)$ with holomorphic quadratic differentials.  These estimates will then be used in the main part of the proof to show that $\na \log\la_1= -\frac{1}{2\la_1} \Rea(P_g^\Hol(\Phi))$ is indeed essentially given by 
$\tfrac{1}{8\pi} \Rea (\tilde \Th^1)$
as described in the theorem.

So let us first consider the case that $\ell_1> \hat \eps=\hat\eps(\hat\de,\gamma)\in (0, 2\min(\hat\de,\arsinh(1)))$. In this case $\inj(M,g)\geq \hat \eps/2$ so the bound \eqref{est:la_rough}
obtained by Schoen-Wolpert-Yau implies that $\la_1$ is bounded away from zero by 
 $\la_1\geq  c_1\hat\eps$. 
The lower bound on the injectivity radius means furthermore that \eqref{est:Linfty-by-L1-1} allows us to 
bound the $L^\infty$-norm of any holomorphic quadratic differential in terms its $L^1$-norm. 
We can thus use the formula for $\na \la_1$ from 
Lemma~\ref{cor:grad-la} to
conclude that
$$  \norm{\na \log(\la_1)}_{L^\infty(M,g)}=C\la_1^{-1} \norm{P_g^\Hol (\Phi(u_1,g))}_{L^\infty(M,g)}\leq C \norm{P_g^\Hol (\Phi(u_1,g))}_{L^1(M,g)}$$
  for a constant $C$ that depends only on $\hat \de$ and the genus of $M$. 
  Moreover, \cite[Proposition 4.10]{RT-neg} implies that 
$
\norm{P^\Hol_g(\Upsilon)}_{L^1(M,g)}\leq C(\gamma) \norm{\Upsilon}_{L^1(M,g)}
$ 
 for any quadratic differential $\Upsilon$. Hence, altogether, we obtain that in this case 
 $$\norm{\na \log(\la_1)}_{L^\infty(M,g)}\leq C\norm{\Phi(u_1,g)}_{L^1(M,g)}\leq C\int_M\abs{d u_1}_g^2dv_g\leq C\la_1,$$
 and thus that the claims of Theorem~\ref{thm:1} hold true for $\al=0$ and $C=C(\hat\de,\gamma)$ as claimed. 

We can thus from now on assume that $\ell_1\leq \hat\eps\leq 2\min(\hat\de,\arsinh(1))$, where $\hat\eps=\hat\eps(\gamma,\hat\de)>0$ is chosen later. We note that this allows us in particular to apply Lemma~\ref{lemma:est-u-thick-main} with $\bar\de=\min(\hat\de,\arsinh(1))$.

In a next step, we now want to combine energy estimates as obtained in Lemma~\ref{lemma:est-u-thick-main} with standard properties of holomorphic quadratic differentials as  recalled in Section~\ref{sec:hol}
 to bound inner products of the Hopf-differential with holomorphic quadratic differentials. 
 These estimates are valid for general eigenfunctions,  though in the present paper will only be applied for $u=u_1$.

  \begin{lemma} \label{lemma:prod-Hopf}
  For any genus $\gamma$ and any number $\bar \de>0$ there exists a constant $C$ so that the following holds true:   
Let $(M,g)$ be any hyperbolic surface and let $u$ be any  eigenfunction of $-\Delta_g$ normalised to  $\Vert u\Vert_{L^2(M,g)}=1$ to an eigenvalue $\lambda$. Let $a_2>0$ be so that
\beq \label{est:u-de-thick-general}
\norm{du}_{L^2(\de\thick(M,g))}^2\leq \frac{a_2}{\de}\lambda^2 \text{ for every } 0< \de \leq  \bar \de
\eeq
Then the Hopf-differential $\Phi(u,g)$ satisfies the following
estimates.
\newline 
For every $\Upsilon\in \Hol(M,g)$ and every $F\subset\bar \de\thick(M,g)$
\beq
\label{est:Hopf-thick}
\abs{\langle \Upsilon, \Phi(u,g)\rangle_{L^2(F,g)}}\leq C a_2\la^2 \norm{\Upsilon}_{L^2(\frac{\bar\de}{2}\thick(M,g))} 
\eeq
while for every simple closed geodesic $\si$ of length 
$\ell=L_g(\si)\leq 2\bar \de$ 
\begin{align}
\label{est:Hopf-decay}
\abs{\langle \Upsilon-b_0(\Upsilon,\Col(\si))dz^2, \Phi(u,g)\rangle_{L^2(\Col(\si))}}&\leq Ca_2\la^2 \norm{\Upsilon}_{L^2(\frac{\bar\de}{2}\thick(M,g))} \\
\abs{\Rea(\langle \i dz^2, \Phi(u,g)\rangle_{L^2(\Col(\si))})}&\leq C\sqrt{a_2}\la^{2}\ell^{-1}.\label{est:Hopf-princ-i}
\end{align}
\end{lemma}

We note that while for general eigenfunctions there  always  exists a sufficiently large number $a_2$ so that \eqref{est:u-de-thick-general} holds true, this number may in general depend on the eigenvalue and other geometric quantities such as $\inj(M,g)$. 

In the setting of Theorem~\ref{thm:1} we know however that \eqref{est:u-de-thick-general} holds true for the constant 
$a_2=C_0(\gamma,\hat \de)$ obtained in Lemma~\ref{lemma:est-u-thick-main} and $\bar \de= 2\min(\hat \de,\arsinh(1))>\ell_1$. Hence in the proofs of our main results we may bound the above inner products of $\Phi(u_1,g)$ simply by   $C\la_1^2 \norm{\Upsilon}_{L^2(\frac{\bar\de}{2}\thick(M,g))} $ respectively by $C\la_1^2$, for a constant $C=C(\gamma,\hat\de)$.

 \begin{proof}[Proof of Lemma~\ref{lemma:prod-Hopf}] 
Let $u$ be a normalised eigenfunction to an eigenvalue $\la$ and let $\bar \de$, $a_2$ be so that \eqref{est:u-de-thick-general} holds true. 
We first note that 
 $\norm{\Phi(u,g)}_{L^1(F,g)}=2\norm{du}_{L^2(F,g)}^2 $ 
 for any subset  $F\subset M$. We hence not only know that 
\beq\label{eq:Phi_L1}
  \norm{\Phi(u,g)}_{L^1(M,g)}=2\la 
  \eeq 
but furthermore get that \eqref{est:u-de-thick-general} implies  
\beq\label{est:Phi_L1-thick}
   \norm{\Phi(u,g)}_{L^1(\bar\de \thick(M,g))}\leq Ca_2\la^2
\eeq
 where here and in the following $C=C(\bar \de, \gamma)$. 
 
 Combined 
 with  \eqref{est:Linfty-by-L1-1} we thus get the first claim of the lemma that for any 
$F\subset \bar \de\thick(M,g)$ and any $\Upsilon\in \Hol(M,g)$
\beqs 
\abs{\langle \Upsilon, \Phi(u,g)\rangle_{L^2(F,g)}}\leq Ca_2\la^2 \norm{\Upsilon}_{L^\infty(\bar \de\thick(M,g))} 
\leq Ca_2 \la^2 \norm{\Upsilon}_{L^2( \frac{\bar\de}{2} \thick(M,g))}.
\eeqs

Let now $\Col(\si)$ be a collar around a simple closed geodesic $\si$ of length $\ell\leq 2\bar \de$ and let $\Upsilon\in\Hol(M,g)$ be any fixed element. We recall that 
the principal part $b_0(\Upsilon)dz^2=b_0(\Upsilon,\Col(\si))dz^2$ of $\Upsilon$ on  $\Col(\si)$ is controlled by \eqref{est:dz2-est} and  \eqref{est:b0-upper}, while on the thick part we can control $\Upsilon$ using \eqref{est:Linfty-by-L1-1}. Combined,  
we may estimate 
\beqa \label{est:proof-prod-Hopf}
& \abs{\langle \Ups-b_0(\Ups)dz^2 ,\Phi \rangle_{L^2(\Col(\si))}}\\
&\qquad \leq \abs{\langle \Ups-b_0(\Ups)dz^2,\Phi \rangle_{L^2(\bar\de\thin(\Col(\si)))}} + 
\abs{\langle \Ups,\Phi \rangle_{L^2(\bar\de\thick(\Col(\si)))} }\\
&\qquad \quad + \abs{b_0(\Ups)}\cdot \norm{dz^2}_{L^\infty(\bar\de\thick(\Col(\si)))} \cdot \norm{\Phi}_{L^1(\bar\de\thick (\Col(\si)))}
 \\
 &\qquad \leq  \abs{\langle \Ups-b_0(\Ups)dz^2,\Phi \rangle_{L^2(\bar\de\thin(\Col(\si)))}} \\
  &\qquad \quad + 
 C \norm{\Phi}_{L^1(\bar\de\thick(M,g))}\cdot \big[\norm{\Upsilon}_{L^\infty( \bar\de \thick(M,g))}+\norm{\Upsilon}_{L^2( \half\arsinh(1) \thick(M,g))}\big]\\
  &\qquad \leq \abs{\langle \Ups-b_0(\Ups)dz^2,\Phi \rangle_{L^2(\bar\de\thin(\Col(\si)))}}+Ca_2 \la^2 \norm{\Upsilon}_{L^2( \frac{\bar \de}{2} \thick(M,g))}
\eeqa
 where we applied \eqref{est:Phi_L1-thick} in the last step. 
To bound the obtained inner product we 
split $\bar \de\thin(\Col(\si))$ into regions of injectivity radius $\inj_g(p)\in[2^{-k-1} \bar\de,2^{-k}\bar\de)$. On such  regions we can bound
$$\norm{\Phi}_{L^1}\leq 2\norm{d u}_{L^2(2^{-k-1}\bar\de\thick(M,g))}^2\leq Ca_2\la^2 (2^{-k}\bar \de)^{-1}$$
using \eqref{est:u-de-thick-general}, while $\Ups-b_0(\Ups)dz^2$ is controlled by \eqref{est:W1}. Combined this gives 
\beqas 
\label{est:inner-prod-phi-decay}
\abs{\langle \Ups-b_0(\Ups)dz^2 ,\Phi \rangle_{L^2(\bar\de\thin(\Col(\si)))}}&
\leq Ca_2\sum_{k\geq 0} e^{-2^k\pi/\bar\de}\cdot (2^{-k}\bar\de)^{-3} \la^2 \norm{\Ups}_{L^2( \bar\de\thick(\Col(\si)))}\\
&\leq  Ca_2\la^2  \norm{\Ups}_{L^2(\bar \de\thick(\Col(\si)))},\eeqas
and inserting this into \eqref{est:proof-prod-Hopf} gives the second claim \eqref{est:Hopf-decay} of the lemma.

To obtain the final claim \eqref{est:Hopf-princ-i} we combine \eqref{def:Hopf-diff} with the angular energy estimate  
\eqref{est:weighted-ang-en-4} that is valid for any eigenfunction and 
 \eqref{est:u-de-thick-general} to conclude that 
\beqas  \label{est:proof-im-inner2}
\abs{\Rea\langle \i dz^2, \Phi(u,g)\rangle_{L^2(\Col(\si))}}&=
2 \abs{\langle \Rea(\Phi),ds\otimes d\th+ d\th\otimes ds\rangle_{L^2(\Col(\si))}}\\
&\leq C \int_{-X(\ell)}^{X(\ell)}\int_{S^1} \abs{u_s} \cdot \abs{u_\theta} \rho^{-2} ds d\theta\\
&
\leq C\bigg(\int_{-X(\ell)}^{X(\ell)}\int_{S^1} \abs{u_s}^2 dsd\th\bigg)^{\half}\bigg(\int_{-X(\ell)}^{X(\ell)}\int_{S^1} \abs{u_\th}^2 \rho^{-4} dsd\th\bigg)^{\half}\\
&\leq C\norm{du}_{L^2(\half \ell\thick(M,g))} \cdot \big[\la^2(1+X(\ell))\big]^{\frac{1}{2}}\\
&\leq C\sqrt{a_2}\ell^{-1} \la^2 .
\eeqas
\vspace{-0.3cm}
\qedhere
\end{proof}

We now prove Theorem~\ref{thm:1} in three steps, establishing first that 
 $\na \log(\la_1)=-\frac{1}{2\la}\Rea(P_g^\Hol(\Phi(u_1,g)))$ is essentially given by the real part of a complex multiple $(\alpha+\i\tilde c)\tilde\Th^1$ of $\tilde \Th^1$, then proving that $\tilde c$ is small, i.e. that the factor $(\al+\i\tilde c)$ is essentially real and finally estimating the size of $\alpha$.

As $\ker(\partial \ell_1)^\perp$ is spanned by $\tilde \Th^1$ 
 we first write 
$$P^\Hol(\Phi(u_1,g))= b\cdot \tilde \Th^1+P_g^{\ker(\partial \ell_1)}(\Phi(u_1,g))
$$
for some $b\in \C$ that is analysed later, and obtain that  $P^\Hol(\Phi(u_1,g))$ is approximately given by $b\cdot \tilde \Th^1$ in the following sense: 

\begin{lemma}\label{lemma:proj-ker-dl}
Suppose that $(M,g)$ is as in Theorem~\ref{thm:1}
with $\ell_1< 2\min(\hat{\de},\arsinh(1))$. 
Then
the orthogonal projection of the Hopf-differential $\Phi(u_1,g)$ of the normalized first eigenfunction $u_1$ onto $\ker(\partial \ell_1)$ is bounded by
$$\norm{P_g^{\ker(\partial \ell_1)}(\Phi(u_1,g))}_{L^{\infty}(M,g)}\leq C\la_1^2$$
where $C$ depends only on the lower bound $\hat{\de}>0$ on $\inj_g\vert_{M\setminus \Col(\si^1)}$ and the genus of $M$.
\end{lemma}

\begin{proof}
We set $w=P_g^{\ker(\partial \ell_1)}(\Phi)$ and recall that 
$b_0(\cdot,\Col(\si^1))=0$ for any element of $\ker(\partial \ell_1)$, so in particular for $w$, compare 
\eqref{def:dell}.
We can thus apply the estimates \eqref{est:Hopf-decay} and \eqref{est:Hopf-thick} of 
Lemma~\ref{lemma:prod-Hopf} (with $F=M\setminus \Col(\si^1)$, $\bar\de=\min(\hat \de, \arsinh(1))$ and $a_2=C_0(\gamma,\hat \de)$)  to obtain 
\beqas
\norm{w}_{L^2(M,g)}^2
&=\langle w,\Phi(u_1,g)\rangle_{L^2(M,g)}= \langle w-b_0^1(w)dz^2,\Phi(u_1,g)\rangle_{L^2(\Col(\si^1))}+\langle w,\Phi(u_1,g)\rangle_{L^2(M\setminus\Col(\si^1),g)}\\
&\leq C\la_1^2 \norm{w}_{L^2(M,g)}.
\eeqas
Combined with 
\eqref{est:W2} this yields that 
$\norm{w}_{L^\infty(M,g)}\leq C\norm{w}_{L^2(M,g)}\leq C\la_1^2$ as claimed. 
\end{proof}

We thus obtain that $\Proj(\Phi)$ is, up to a well controlled error term, a complex multiple $b\tilde \Th^1$ of  $\tilde \Th^1$. 
In a next step we show that this factor $b$ is almost real which, as we shall see later on, is crucial to prove that Dehn-twists on $\Col(\si^1)$ do not have a significant effect on the first eigenvalue. As $\tilde \Th^1$ is a real multiple of the renormalised element $\tilde \Om^1 $ this will follow from

\begin{lemma}\label{lemma:Im-Th1}
Let $(M,g)$ be as in Theorem~\ref{thm:1} with $\ell_1< 2\min(\hat \de, \arsinh(1))$ and let $\tilde\Om^1$ be as in Definition~\ref{defn:basis1}. Then the Hopf-differential of the normalised first eigenfunction $u_1$ satisfies
$$\abs{\Ima\langle P_g^\Hol(\Phi(u_1,g)),\tilde \Om^1\rangle_{L^2(M,g)}} \leq C\ell_1^{1/2}\la_1^2, \text{ where } C=C(\hat{\de},\gamma).$$
\end{lemma}

\begin{proof}[Proof of Lemma~\ref{lemma:Im-Th1}]
We note that 
$\langle P_g^\Hol(\Phi(u_1,g)),\tilde \Om^1\rangle_{L^2(M,g)}=\langle \Phi(u_1,g),\tilde \Om^1\rangle_{L^2(M,g)}$ is essentially given by the inner product of $\Phi=\Phi(u_1,g)$  and the principal part $b_0^1(\tilde \Om^1)dz^2$ of $\tilde \Om^1$ on $\Col(\si^1)$; to be more precise,
combining Lemma~\ref{lemma:prod-Hopf} with \eqref{est:Om-tilde-thick} 
yields 
\beqas 
\abs{ \langle \Phi,\tilde \Om^1\rangle_{L^2(M,g)}
-\langle \Phi,b_0^1(\tilde \Om^1)dz^2\rangle_{L^2(\Col(\si^1))}}
& = \abs{\langle \Phi,\tilde \Om^1\rangle_{L^2(M\setminus \Col(\si^1))}+\langle \Phi,\tilde \Om^1-b_0^1(\tilde \Om^1)dz^2\rangle_{L^2(\Col(\si^1))}}\\
&
\leq C \norm{\tilde \Om^1}_{L^2(\frac{\bar\de}{2}\thick(M,g))}\cdot \la_1^2 
\leq C\ell_1^{3/2}\la_1^2
\eeqas 
for $\bar\de=\min(\hat\de,\arsinh(1))$ and a constant $C=C(\gamma, \hat \de)$. 

Since the principal part of $\tilde \Om^1$ on $\Col(\si^1)$ is real, Lemma~\ref{lemma:prod-Hopf} furthermore gives 
\beqas 
\abs{\text{Im}\langle \Phi,b_0^1(\tilde\Om^1)dz^2\rangle_{L^2(\Col(\si^1))}} &= \abs{b_0^1(\tilde\Om^1)}\cdot \abs{\text{Re} \langle\Phi,\i dz^2\rangle_{L^2(\Col(\si^1))}}
\leq C\ell_1^{3/2} \cdot C\la_1^2\ell_1^{-1}
\leq C\ell_1^{1/2}\la_1^2
\eeqas
where we used \eqref{est:b0-trivial-small} in  the penultimate step. 
 Combined  this yields the claim of the lemma. 
\end{proof}

At this stage we thus know 
that we can write 
\beq 
\na \log(\la_1) = - P^{\Rea(\Hol)}_g(\Rea(\frac1{2\la_1}\Phi(u_1,g))= \Rea(P^\Hol_g(-\frac{1}{2 \la_1}\Phi(u_1,g)))
= \Rea(\al \tilde \Th^1 +R )\label{eq:writing-grad}
\eeq
for a real number $\al$
and an error term of the form 
\beq \label{eq:error-R}
R= \i c_0\tilde \Om^1-\frac{1}{2 \la_1}P_g^{\ker(\partial \ell_1)}\Phi(u_1,g), \text{ for some } c_0\in \R,\eeq
where we note that
$\abs{c_0}=\frac{1}{2\la_1}\abs{\Ima\langle \Phi, \tilde \Om^1\rangle}\leq C\ell_1^{1/2} \la_1 $ thanks to Lemma~\ref{lemma:Im-Th1}.
Since $\norm{\tilde\Om^1}_{L^\infty(M,g)}\leq C\ell_1^{-1/2}$ by \eqref{est:RT-neg3}, while 
the second term in the above estimate is controlled by Lemma~\ref{lemma:proj-ker-dl}, we have  
\beq 
\label{est:error-R}
\norm{R}_{L^\infty(M,g)} \leq C\la_1\leq C\ell_1.\eeq
This establishes the claim \eqref{claim:na-la-2} of the theorem. 

To prove the remaining claim \eqref{claim:alpha} of
Theorem~\ref{thm:1} we now show that the coefficient $\al$ in \eqref{eq:writing-grad} satisfies
$\abs{\al}=\frac{1}{8\pi}+O(\ell_1\log(\ell_1))$ and, in a later step, that for sufficiently small $\ell_1$ also $\al>0$.

We recall from \eqref{eq:Phi_L1} that $\norm{ \Phi}_{L^1(M,g)}=2\la_1$,  while 
$\norm{\tilde\Th^1}_{L^1(M,g)}=8\pi+O(\ell_1)$ by   \eqref{est:L2-tilde-theta}. 
So as \eqref{eq:writing-grad} and \eqref{est:error-R} imply that 
$\al\tilde\Th^1+R=-\frac{1}{2\la_1} P_g^\Hol(\Phi)$ with $\norm{R}_{L^\infty}=O(\ell_1)$,  we get
\beqas
\abs{\al}& 
=\norm{\tilde \Th^1}_{L^1(M,g)}^{-1}\big[\tfrac{1}{2\la_1} \norm{ P^\Hol_g(\Phi)}_{L^1(M,g)} +O(\ell_1)\big]=\frac{1}{ (8\pi+O(\ell_1))}\cdot 
 \frac{\norm{P^\Hol_g(\Phi)}_{L^1(M,g)}}{\norm{ \Phi}_{L^1(M,g)} }+O(\ell_1).
\eeqas
To obtain the desired bound on $\abs{\al}$, it thus suffices to prove that
$$\babs{\norm{P_g^\Hol(\Phi)}_{L^1(M,g)}-
\norm{\Phi}_{L^1(M,g)}
} \leq C \ell_1^2\abs{\log(\ell_1)}.$$
As \eqref{est:la_rough} implies that $\la_1 \leq C\ell_1$, this follows from 

\begin{lemma}
\label{lemma:est-Phi-hol}
Let $(M,g)$ be as in Theorem~\ref{thm:1} with $\ell_1< 2\bar \de:= 2\min(\hat \de, \arsinh(1))$. Then the Hopf-differential $\Phi=\Phi(u_1,g)$ of the normalised first eigenfunction $u_1$ satisfies
\beqs \label{est:Phi-minus-hol}
\norm{\Phi-P_g^\Hol(\Phi)}_{L^1(M,g)}\leq C\la_1^2 \abs{\log(\ell_1)}, \text{ for some } C=C(\gamma, \hat{\de}) .
\eeqs
\end{lemma}

The crucial ingredient in the proof of this lemma is the following uniform Poincar\'e estimate for quadratic differentials from the joint work \cite{RT2} of P. Topping and the second author

\begin{thm}[Theorem~1.1 of \cite{RT2}]\label{thm:Poincare}
For any genus $\gamma\geq 2$ there exists a constant $C_\gamma<\infty$ so that for any closed oriented hyperbolic surface $(M,g)$ 
of genus $\gamma$
the distance of any  
$L^2$-quadratic differential $\Upsilon$ from its holomorphic part is bounded by  
\beqs 
\label{est:Poincare}
\norm{\Upsilon-P_g^\Hol(\Upsilon)}_{L^1(M,g)}\leq C_\gamma\norm{\bar\partial \Upsilon}_{L^1(M,g)}.
\eeqs
\end{thm}
We note that it is crucial for our application that $C_\gamma$ is a topological constant, depending only on the genus and not on geometric quantities such as the diameter of $(M,g)$. 


\begin{proof}[Proof of Lemma~\ref{lemma:est-Phi-hol}]
To derive the lemma from Theorem~\ref{thm:Poincare}
we need to prove that $\Phi=\Phi(u_1,g)$ is \textit{almost holomorphic} in the sense that the estimate
\beqs 
\label{est:d-bar-Phi}
\norm{\bar\partial \Phi}_{L^1(M,g)}\leq C\la_1^2\abs{\log(\ell_1)}
\eeqs
holds true for a constant $C$ that depends only on the genus and on $\hat{\de}$.  

To this end we recall that the antiholomorphic derivative of the Hopf-differential of maps from a surface to an arbitrary Riemannian manifolds is bounded in terms of the tension field, so in our situation simply by the Laplacian. To be more precise, working in local isothermal coordinates $(x,y)$, $z=x+\i y$, we may write 
$\bar \partial\Phi=\frac{1}{2}\left( \partial_x\phi +\i \partial_y\phi \right) d\bar z\otimes dz^2= (\partial_{xx} u_{1}+\partial_{yy} u_{1})\cdot (\partial_x u_1-\i \partial_y u_1) d \bar z\otimes dz^2$. Combined with 
\eqref{eq:dz-rho} we thus get
$$\abs{\bar \partial\Phi}_g\leq \rho^2 \abs{\Delta_g u_1} \cdot \rho \abs{du_1}_{g} \cdot \abs{dz}_g^3= 2\sqrt{2} \abs{\Delta_g u_1} \cdot\abs{du_1}_{g}.$$
Since our first eigenfunction $u_1$ is uniformly bounded, c.f. Remark~\ref{rmk:L-infty}, with $\norm{u_1}_{L^2}=1$ we thus have 
\beqa \label{est:bar-Phi-proof}
\norm{\bar\partial \Phi}_{L^1(M,g)}& 
\leq C\la_1 \int \abs{u_1}\cdot \abs{du_1}_g dv_g \leq C\la_1 \bigg[ \norm{d u_1}_{L^2(M\setminus \Col(\si^1))}
+ \norm{u_1}_{L^\infty(M,g)}\int_{\Col(\si^1)}\abs{du_1}_g dv_g\bigg]\\
&\leq C\la_1^2+ C\la_1 \int_{\Col(\si^1)}\abs{du_1}_g dv_g
\eeqa
where we used the energy estimate \eqref{est:u-de-thick} of Lemma~\ref{lemma:est-u-thick-main} on $M\setminus \Col(\si^1)\subset \bar{\de}\thick(M,g)$ in the last step.
To obtain a bound of $C\ell_1\abs{\log(\ell_1)}$ for the last integral instead of just the trivial bound of $C\norm{du_1}_{L^2}\leq C\ell_1^{1/2}$, we note that 
$du_1$ is small near the ends of the collar while regions near the centre of the collar have small volume. 
We thus split the collar into subsets 
$$C_k:=\{p\in\Col(\si^1): 2^{k-1}\ell_1\leq \inj_g(p)<2^k\ell_1\} \quad 0\leq k\leq \bar K$$
whose total number is bounded by $\abs{\bar K}\leq C  \abs{\log(\ell_1)} $ as  $\inj_g$ is bounded from above uniformly.
Combining the bound on  $\text{Area}_g(C_k)\leq \text{Area}_g(2^k\ell_1\thin(\Col(\si^1)))\leq C 2^k \ell_1$ from \eqref{est:area-thin}  with 
 Lemma~\ref{lemma:est-u-thick-main} gives that for every $k$
\beqs \label{est:split_du}
\int_{C_k} \abs{du_1}_{g} dv_g \leq \text{Area}_g(C_k)^{1/2}\norm{d u_1}_{L^2(2^{k-1}\ell_1\thick(\Col(\si^1)))}\leq C\cdot (2^k\ell_1)^{1/2}\cdot \big(\frac{1}{2^k\ell_1} \la_1 ^2)^{1/2}\leq C\la_1. 
\eeqs
Thus
\eqref{est:bar-Phi-proof} reduces to 
$$
\norm{\bar\partial \Phi}_{L^1(M,g)}\leq C\la_1^2+C\bar K \la_1^2 \leq C\la_1^2 \abs{\log(\ell_1)}
$$
which is the bound that we needed to derive the lemma from Theorem~\ref{thm:Poincare}.
\end{proof}

Having thus established that the coefficient $\al$ in \eqref{eq:writing-grad} is so that $\abs{\al}=\frac{1}{8\pi}+O(\ell_1\log(\ell_1))$ we finally complete the proof of Theorem~\ref{thm:1} by showing that $\al>0$ for sufficiently small $\ell_1$.  
We recall that $\al$ is defined by 
\eqref{eq:writing-grad} and \eqref{eq:error-R} and hence characterised by
$$\al=-\tfrac{1}{2\la_1} \Rea\lan \Phi(u_1,g), \tfrac{\tilde \Th^1}{\norm{\tilde \Th^1}_{L^2}^2}\ran.$$
As the first eigenvalue is simple, we
know that the normalised first eigenfunction depends continuously on the
metric, and one can easily check that also $g\mapsto \tilde \Th^1(g)$ is continuous. Hence $\al$ itself depends continuously on the 
Fenchel-Nielsen coordinates. As $\abs{\al}$ is bounded away from zero for sufficiently small $\ell_1$, it must thus have constant sign for $\ell_1\in (0,\hat \eps)$, for a sufficiently small number $\hat \eps =\hat\eps(\hat\de,\gamma)>0$. 

It hence remains to exclude the possibility that $\al<0$ for all $\ell_1\in (0,\hat\eps)$. As we shall see in \eqref{est:al-needed} we have that 
$\abs{\tfrac{\partial \la_1}{\partial \ell_1}- \al \tfrac{8\pi\la_1}{\ell_1}}\leq  C\ell_1^3$. At the same time the results of Burger imply that 
$\frac{\la_1}{\ell_1}\to C_{top}$. So if $\al$ was negative, and hence $\al=-\frac{1}{8\pi}+O(\ell_1\log(\ell_1))$, we would obtain that $\tfrac{\partial \la_1}{\partial \ell_1}=-1+O(\ell_1\log(\ell_1))<0$ for small values of $\ell_1$. This would of course lead to a contradiction as $\la_1>0$ with $\la_1\to 0$ as $\ell_1\to 0$. Hence indeed $\al=\frac{1}{8\pi}+O(\ell_1\log(\ell_1))$ as claimed in the theorem.

\subsection{Proof of Theorem~\ref{thm:FN}}
\label{subsec:proof_thmFN}
$ $

As the eigenvalues are invariant under pull-back of the metrics by diffeomorphisms, we know from Section~\ref{sec:dual} that the derivatives of any simple eigenvalue with respect to the Fenchel-Nielsen coordinates are given by 
\beq\label{eq:deriv-evs-FN}
\frac{\partial \la}{\partial\psi_j}=\langle \na \la, \Rea(\Psi^j )\rangle  \text{ and } \frac{\partial \la}{\partial\ell_j}= \langle \na \la, \Rea(\La^j )\rangle.
\eeq 
Here and in the following $\{\La^j,\Psi^j\}$ denotes the dual basis to the real differentials of the Fenchel-Nielsen coordinates that was defined in Definition~\ref{defn:basis1} and  we recall that $\La^j,\Psi^j$ can be described in terms of the dual bases $\{\Om^j\}$ and $\{\Th^j\}$ of the complex differentials $\partial \ell_j$ as explained in Proposition~\ref{prop:La-Psi}. 

As the previous section gives a characterisation of $\na \la_1$  in terms of $\tilde \Th^1$ from Definition~\ref{def:Th-tilde}, we first use Theorem~\ref{thm:1} to derive a closely related expression for $\na \la_1$ in terms of the bases 
$\{\Th^j\}$ and $\{\Om^j\}$ from Definition~\ref{defn:basis1}. 

To be more precise, we claim that 
 the results of the previous section imply that
\beq \label{eq:writing_na_la_new}
\na \la_1= \al \la_1 \Rea(\Th^1)+ \tilde c\cdot \Rea( \i \Om^1) +\sum_{k\geq 2}\Rea(d_k \Om^k) \eeq
for the same  $\al$ as obtained in the proof of Theorem~\ref{thm:1}
and coefficients
\beq \label{est:coeff_proof}
\tilde c\in \R \text{ satisfying } \abs{\tilde c}\leq C \ell_1^{5/2} \text{ as well as } d_j\in \C \text{ with } \abs{d_j}\leq C\ell_1^2.\eeq

Here and in the following $C$ is allowed to depend on $\hat\de$ and the genus (and so also the upper bound $\bar L=\bar L(\gamma,\ell_0(\hat\de, \gamma))$ on all $\ell_j$).

To see that $\na \la_1$ is of the above form we recall from
\eqref{eq:writing-grad} and \eqref{eq:error-R} 
 that 
$$ \na \la_1= \al \la_1 \Rea(\tilde\Th^1)+ \Rea(\i \la_1 c_0\tilde \Om^1 +w) $$
for some  $c_0\in \R$ with $\abs{c_0}\leq C\ell_1^{1/2}\la_1 \leq C\ell_1^{3/2}$ and $w=-\frac12P^{\ker(\partial \ell_1)}_g(\Phi)$. We also recall that Lemma~\ref{lemma:proj-ker-dl} implies that $\norm{w}_{L^\infty(M,g)}\leq C\la_1^2\leq C\ell_1^{2}$. 

We then use Proposition~\ref{prop:Th-tildeTh}  to write  $\tilde \Th^1=\Th^1-v^1$ and $\tilde \Om^1=\beta_1^{-1} (\Om^1-w^1)$ for elements $w^1,v^1$ of $\ker(\partial \ell_1)$ with $\norm{v^1}_{L^\infty} \leq C\ell_1$ and $\norm{w^1}_{L^\infty}\leq C\ell_1^{3/2}$ 
and a coefficient $\beta_1\in \R^+$ with $\beta_1^{-1}\leq \eps_2^{-1}\leq C$. As $\ker(\partial \ell_1)$ is spanned by $\{\Om^j\}_{j\neq 1}$ we may thus write 
 $\na \la_1$ in the form \eqref{eq:writing_na_la_new} for the same number $\al$ as obtained in the proof of Theorem~\ref{thm:1} and 
a number $\tilde c\in \R$ which is bounded by 
$\abs{\tilde c}\leq C \cdot \abs{c_0}\la_1\leq C\ell_1^{5/2}$. Furthermore, the coefficients $d_k\in \C$  must be so that 
$$\norm{\sum_{j\geq 2} d_k \Om^k}_{L^\infty(M,g)}
\leq C\la_1(\abs{c_0}\norm{w^1}_{L^\infty(M,g)}+\norm{v^1}_{L^\infty(M,g)})+\norm{w}_{L^\infty(M,g)}
\leq 
C\ell_1 \cdot \la_1+\ell_1^2\leq C\ell_1^2,$$ which, by  \cite[Remark 2.10]{holo-paper}, implies that 
$
\abs{d_k}\leq C\ell_1^2$
for every $k$ as claimed.

We now estimate the derivatives of $\la_1$ with respect to the Fenchel-Nielsen coordinates by combining the above expressions \eqref{eq:deriv-evs-FN} and \eqref{eq:writing_na_la_new} 
 with the bounds on $\Om^j$, $\La^j$ and $\Psi^j$ and their inner products from Section~\ref{sec:dual}. 
 
We begin by analysing the 
derivatives of $\la_1$ with respect to the twist coordinates, i.e. by estimating $\langle \na \la_1,\Rea(\Psi^j)\rangle$. Here we crucially use that Wolpert's twist-length duality assures that $\Psi^j$ is  orthogonal to $\ker(\partial \ell_j)=\text{span}\{\Om^i\}_{i\neq j}=\text{span}\{\Th^i\}_{i\neq j}$, compare  \cite[Theorem 2.10]{Wolpert82}, so that the inner product of $\Rea(\Psi^j)$ with most terms in \eqref{eq:writing_na_la_new} vanishes. 

Hence, if $j\geq 2$, we 
obtain the claimed bound of 
\beqs 
\babs{\frac{\partial \la_1}{\partial\psi_j}}= \babs{\langle \na \la_1, \Rea(\Psi^j )\rangle}=  \babs{\langle  \Rea(d_j\Om^j), \Rea(\Psi^j)\rangle}\leq C \norm{\Psi^j}_{L^2} \abs{d_j}\leq C\ell_j^{3/2}\ell_1^2\leq C\ell_1^2,
\eeqs
where we use \eqref{est:coeff_proof} and \eqref{est:L2-Psi-upper} in the penultimate step. Here and in the following 
all norms and inner products computed over all of $(M,g)$.

For $j=1$ we obtain by the same argument that
\beqas 
\babs{\frac{\partial \la_1}{\partial \psi_1}}&= \norm{\Psi^1}_{L^2} \cdot\babs{ 
\al\la_1 \langle \Rea(\Th^1), \Rea(\tfrac{\Psi^1}{\norm{\Psi^1}_{L^2}})\rangle + \tilde c \langle \Rea(\i\Om^1), \Rea(\tfrac{\Psi^1}{\norm{\Psi^1}_{L^2}})\rangle }\\
&\leq C\ell_1^{3/2} \cdot\bigg[\ell_1 \norm{\Th^1}_{L^2} \cdot 
\abs{\langle \Rea(\Om^1), \Rea(\tfrac{\Psi^1}{\norm{\Psi^1}_{L^2}})\rangle}
+\abs{\tilde c} \bigg]
\eeqas
where we use 
\eqref{est:L2-Psi-upper} as well as $\la_1\leq C\ell_1$ in the second step. 

We recall that 
$\norm{\Th^1}_{L^2}\leq C\ell_1^{-1/2}$, see~\eqref{est:Th-L2-upper}, and that 
the above inner product is controlled by the estimate \eqref{est:inner-Om-Psi-same} of Lemma~\ref{lemma:inner-prod}. Together with the bound on $\tilde c$ from \eqref{est:coeff_proof} we hence obtain 
$$\babs{\frac{\partial \la_1}{\partial \psi_1}}
\leq  C\ell_1^{3/2} \cdot[\ell_1 \ell_1^{-1/2} \ell_1^3 
+\ell_1^{5/2} ]\leq C\ell_1^4$$
as claimed in the theorem.

We now turn to the proof of the bounds on the derivatives of $\la_1$ with respect to the length coordinates. 
As $\na\la_1$ is described by \eqref{eq:writing_na_la_new} we know that  
\beq \label{eq:expr-dla}
\frac{\partial \la_1}{\partial \ell_j} =
\langle \na \la_1, \Rea\La^j\rangle=  \al \la_1 \cdot \langle \Rea \Th^1,\Rea\La^j\rangle+R_j 
\eeq
for a remainder term $R_j$ which, thanks to \eqref{est:coeff_proof}, is bounded by 
\beqa \label{est:Rj}
\abs{R_j} &\leq \abs{\tilde c}\cdot \abs{ \langle \Rea\i\Om^1,\Rea\La^j\ran}+\sum_{k\geq 2} \abs{ \langle \Rea (d_k\Om^k),\Rea\La^j\ran}\\
&\leq 
C\ell_{1}^{5/2} \abs{\langle \Rea\i\Om^1,\Rea\La^j\rangle} +C\ell_1^2 \cdot \max_{k\geq 2} \abs{\lan \Om^k,\La^j\rangle}.
\eeqa
 We first consider the case that 
$j\neq 1$. In this case $\ell_j\geq 2\bar \de$ which, by \eqref{est:La-minus-Th}, implies that $\norm{\La^j}_{L^2}\leq C$ and thus that the second term in  \eqref{est:Rj} is bounded by $C\ell_1^2$. To bound the first term 
we can apply the estimate \eqref{est:inner-Om-La-diff} of  Lemma~\ref{lemma:inner-prod} and hence obtain that for $j\neq 1$ 
 $$\abs{R_j}\leq C \ell_1^{5/2} \ell_j\ell_1^{3/2}+C\ell_1^2\leq C \ell_1^2. $$

Furthermore as $\la_1\leq C\ell_1$ and $\norm{\Th^1}_{L^2}\leq C\ell_1^{-1/2}$ the same  estimate \eqref{est:inner-Om-La-diff} 
implies that for $j\neq 1$ also the main term in 
 \eqref{eq:expr-dla} 
is controlled by 
 $$\abs{\al \la_1 \langle \Rea \Th^1,\Rea\La^j\rangle}\leq C\ell_1 \norm{\Th^1}_{L^2} \abs{\langle \Om^1, \La^j\rangle} \leq C\ell_1 \ell_1^{-1/2}\ell_j\ell_1^{3/2}=C\ell_1^2.
 $$
 Combined we thus obtain the claimed bound of $\abs{\frac{\partial \la_1}{\partial \ell_j}}\leq C\ell_1^2$ for $j\neq 1$. 
 
Finally, let  $j=1$. Then the remainder term from \eqref{est:Rj} can again be bounded using Lemma~\ref{lemma:inner-prod}, now using both \eqref{est:inner-Th-La-same-Im} and \eqref{est:inner-Om-La-diff} to get  
$$\abs{R_1}\leq  C\ell_1^{5/2}\ell_1^{5/2}+C\ell_1^2\ell_1\max_{k\geq 2} \ell_k^{3/2} \leq C\ell_1^3.$$

To analyse the main term in \eqref{eq:expr-dla} in case $j=1$ we note that the
estimate \eqref{est:inner-Th-La-same-Re} of Lemma~\ref{lemma:inner-prod} combined with \eqref{est:L^1-the1-precise} 
implies that 
$\abs{ \langle \Rea \Th^1,\Rea\La^1\ran-\frac{8\pi}{\ell_1}}\leq C\ell_1^2$. 
Combined with the above bound on $R_1$ we thus conclude that 
\beqa \label{est:al-needed}
\abs{\tfrac{\partial \la_1}{\partial \ell_1}- \al\la_1 \tfrac{8\pi}{\ell_1}}\leq C\abs{\al} \la_1 \ell_1^2+\abs{R_1}\leq C\ell_1^3.
\eeqa

We note that up to this point we have only ever used that $\al\in\R$ is bounded uniformly, and have not used that $\al\geq 0$, which justifies the application of the above estimate in the last part of the proof of Theorem~\ref{thm:1} where we show that $\al\geq 0$. 

Finally using that $\al=\frac{1}{8\pi}+O(\ell_1\log(\ell_1))$ we obtain that indeed 
$$\abs{\tfrac{\partial \la_1}{\partial\ell_1}-  \tfrac{\la_1}{\ell_1}}\leq C\ell_1 \abs{\log(\ell_1)} (1+ \tfrac{\la_1}{\ell_1})\leq  C\ell_1 \abs{\log(\ell_1)} ,$$ completing the proof of Theorem~\ref{thm:FN}. 

\subsection{Proof of Corollary~\ref{thm:C0-est}}\label{subsec:cor}
$ $

Let $M$ be a closed oriented surface of genus $\gamma\geq 2$ and let $\bar \si$ be a simple closed curve that disconnects $M$ into two connected components.  Let $\mathcal{\bar E}_1=\{\bar \si_1^j\}_{j=1}^{3(\gamma-1)}$ be a collection of disjoint simple closed curves which decomposes $M$ into pairs of pants and which is chosen so that $\bar\si_1^1=\bar \si$.

We note that there exists a finite set of decomposing collections $\mathcal{\bar E}_i=\{\bar \si_i^j\}_{j=1}^{3(\gamma-1)}
$, $i=2,\ldots, N$, of disjoint simple closed curves in $M$, 
 with $\bar{\si}_i^1= \bar{\si}$ for every $i$, such that the following holds true:
 For any  decomposing collection $\tilde{\mathcal{E}}=\{ \tilde{\si}^1, \ldots, \tilde \si^{3(\gamma -1)}\}$ of disjoint simple closed curves in $M$ for which $\tilde \si^1$ and $\bar{\si}$ are homotopic,  there exists an index $i\in \{1,\ldots, N\}$ and a diffeomorphism $\tilde f\colon M\to M$
which maps $\tilde \si^j$ to $ \bar{\si}_i^j$ for every $j=1,\ldots, 3(\gamma-1)$.

This well-known property can be seen as follows: After cutting the surface along the curves $ \bar{\si}$ respectively $\tilde \si^1$,
we obtain two surfaces $\Si_1$ and $\Si_2$ each having one boundary curve. The number of decomposing collections $\mathcal{E}_i$ then corresponds to the number of different ways  that the sets $\Si_1$ and $\Si_2$ can be built from pairs of pants (some containing the boundary curve of $\Si_i$) while keeping track of which boundary curves are glued together.

We now introduce Fenchel-Nielsen coordinates associated with one of these collections, say with $\mathcal{\bar E}_1$, 
and let $(g_\ell)_{\ell\in (0,2\arsinh (1))}$ be a family of hyperbolic metrics on $M$ for which the first Fenchel-Nielsen length coordinate is $\ell_1=\ell\in (0,2\arsinh(1))$ while all other Fenchel-Nielsen coordinates are given by fixed numbers  $\ell_j\equiv c_j$ and $\psi_j\equiv \tilde c_j\in[0,2\pi]$.

We denote the geodesics in $(M,g_\ell)$ that are homotopic to the curves $\bar \si_i^j$ by $\si_i^j(\ell)$ and let 
$\mathcal{E}_i(\ell)=\{\si_i^j(\ell)\}_{j=1}^{3(\gamma-1)}$. We furthermore write for short 
$ \si(\ell)=\si_i^1(\ell)$ for the geodesic that is homotopic to $\bar \si$ (and which thus has length $\ell$) and denote by $\Col(\si(\ell))$ the corresponding collar in $(M,g_\ell)$.

To begin with we claim that on $M\setminus \Col(\si(\ell))$
the injectivity radius $\inj_{g_\ell}$ 
  is bounded away from zero by a constant $\de_0>0$ that depends only on the genus 
  (having fixed the numbers $c_j$):
  To see this we recall that if the injectivity radius in a point $p\in (M\setminus \Col(\si(\ell)),g_\ell)$ is equal to some $\de\in (0,\arsinh(1))$ then this point must be in a collar around a geodesic $\tilde \si\subset (M,g_\ell)$  of length no more than $2\de$. 
  This geodesic either agrees with one of the $\si_1^j(\ell)\in \mathcal{E}_1(\ell)$, $j\neq 1$, in which case $\de\geq \half \min(c_j)$, or it must intersect at least one of the $\si_1^j(\ell)\in\mathcal{E}_1(\ell)$  (as a pair of pants does not contain any simple closed geodesics), in which case its length is bounded below by the width $w_{\ell_j}$ of the corresponding collar, 
which is related to $\ell_j$ by $
\sinh \tfrac{w_{\ell_j}}{2} \sinh \tfrac{\ell_j}{2}=1$. Hence in this 
second case either $\de\geq 2\arsinh(1)$  (if $j=1$) or  $\de\geq \frac{w_{\ell_j}}{2}=\arsinh((\sinh(\frac{c_j}{2}))^{-1})$, and so in this case $\de$ is bounded away from zero in terms of $\max(c_j)$.

We also note that since the collections $ \mathcal{\bar E}_i$ (and the Fenchel-Nielsen coordinates  $\ell_j=c_j$ and $\phi_j=\tilde c_j$ with respect to $ \mathcal{\bar E}_1$) are fixed  we also have an upper bound $\bar L=\bar L(\gamma)$ on the lengths 
of all geodesics 
$\si_i^j(\ell)$ in $(M,g_\ell)$ which are homotopic to one of the simple closed curves in $\bigcup_i \mathcal{\bar E}_i$. As a result, there exist numbers $\eta>0$ and $\bar L$ depending only on the genus so that the usual assumptions \eqref{ass:eta} and \eqref{ass:upperbound} are satisfied for 
 each of the metrics $g_\ell$ and each of the associated collections $\mathcal{E}_i(\ell):=\{\si_i^j(\ell)\}_{j=1}^{3(\gamma-1)}$, $i=1,\ldots, N(\gamma)$, of simple closed geodesics. 
 
These observations allow us to derive Corollary~\ref{thm:C0-est} from Theorem~\ref{thm:FN} as follows. 
Let 
$$f(\ell):=\la_1(M,g_\ell)$$
and note that Theorem~\ref{thm:FN}, applied for the Fenchel-Nielsen coordinates associated with the collection $\mathcal{\bar E}_1$ for which all coordinates except $\ell_1$ are constant along $(g_\ell)_\ell$,  yields that 
$$\babs{\frac{d}{d\ell}\frac{f(\ell)}{\ell}}= 
 \ell^{-1}\babs{\frac{\partial \la_1}{\partial \ell_1}-\frac{\la_1}{\ell}}\leq C\  \abs{\log(\ell)} \text{ for } 0<\ell<\ell_0,$$
where $\ell_0=\ell_0(\de_0, \gamma)$ is as in Remark~\ref{rmk:la-simple} and depends only on the genus. 
 
 Using that the result \eqref{est:burger} of Burger implies in particular
 that 
 $\frac{f(\ell)}{\ell}$ converges to the constant $C_{top}$ defined in \eqref{def:Ctop} as $\ell\to 0$,  we can thus integrate this bound to  
 $$\babs{\frac{f(\ell)}{\ell}-C_{top}}\leq C\ell\babs{\log(\ell)}$$
 holds true, initially for $\ell\in (0,\ell_0)$, and as this estimate trivially holds true for larger values of $\ell$, thus  
indeed for $\ell\in (0,2\arsinh(1))$ as claimed.

Let now $g$ be any hyperbolic metric on $M$ which satisfies the assumptions of the corollary. 
We extend the geodesic $\si$ in $(M,g)$ to a disconnecting set  $\{\hat \si^j\}_{j=1}^{3(\gamma-1)}$, with $\hat \si^1=\si$, of simple closed geodesics
which we recall can be chosen so that $L_g(\hat\si^j)\leq \bar L=\bar L(\gamma)$, compare Lemma~\ref{lemma:appendix-collect}. 
As explained at the beginning of the section, we can now choose $i\in \{1,\ldots,N(\gamma)\}$ so that there exists a 
diffeomorphism $\tilde f\colon M\to M$ which maps $\bar \si_i^j\in\mathcal{\bar E}_i$ to $\hat \si^j $ for every $j=1,\ldots, 3(\gamma-1)$. 
We also note that this diffeomorphism can be chosen so that the twist coordinates (with respect to $\mathcal{\bar E}_i$) of the resulting metric $\tilde f^*g$ are in $[0,2\pi]$.

We then consider the Fenchel-Nielsen coordinates $(\tilde \ell_j, \tilde \psi_j)$ associated to 
$\mathcal{\bar E}_i$ of both $\tilde f^*g$ and the element $g_{\ell=L_g(\si)}$ of the family of metrics considered above. 
We recall that the length coordinates $\ell_j$, $j\geq 2$, of $g_\ell$ are bounded away from zero by the constant 
$2\de_0(\gamma)$ obtained above, while the assumption of the corollary implies that the length coordinates $\ell_j$, $j\geq 2$, of  $\tilde f^*g$ are at least 
 $2\hat \de$. Furthermore, by construction, the length coordinates of both of these metrics are bounded 
from above by a constant $\bar L$ that depends only on the genus. 
We thus also obtain bounds of $2\min(\de_0,\hat\de)\leq \ell_j\leq \bar L$ on the length coordinates $\ell_j$, $j\geq 2$, of the metrics $(g(t))_{t\in[0,1]}$ which interpolate between $\tilde f^*g$ and $g_\ell$ in the sense that their Fenchel-Nielsen coordinates are $\tilde \ell_j(g(t))=t\tilde\ell_j(\tilde f^*g)+(1-t)\tilde \ell_j(g_\ell)$ and likewise for the twist coordinates.

Arguing as in the first part of the proof we can hence obtain a uniform lower bound $\hat \de_0$ on the injectivity radius $\inj_{g(t)}$ on $M\setminus \Col(\si(t))$ 
 in terms of $\max_{j\geq 2}(\tilde\ell_j(f^*g),\tilde\ell_j(g_\ell))\leq \bar L(\gamma)$ and $\min_{j\geq 2}(\tilde\ell_j(\tilde f^*g),\tilde \ell_j(g_\ell)) \geq \min(2\hat \de, 2\de_0)$, 
 where $\si(t)$ is the unique geodesic in $(M,g(t))$ homotopic to $\bar\si$.

This allows us to now complete the proof of the second claim \eqref{claim:la-f} of the corollary as follows: 

Let $\ell_0=\ell_0(\hat{\de}_0, \gamma)>0$ be the constant from Remark~\ref{rmk:la-simple}, where $\hat \de_0=(\hat{\de}_0,\gamma)$ is as obtained above. We note that this constant $\ell_0$ depends only 
on the assumed lower bound $\hat \de$ on $\inj_g $ on $M\setminus \Col(\si)$. 
In case that $\ell\geq \ell_0$ we hence have that \eqref{claim:la-f} is trivially true provided the constant $C=C(\hat\de,\gamma)$ is chosen sufficiently large. 
Conversely, in case $\ell\leq \ell_0$ we know that 
the assumptions of Theorem~\ref{thm:FN} 
hold true for every $g(t)$, $t\in [0,1]$ 
(with $\hat \de$ replaced by $\hat \de_0$). We can thus apply \eqref{claim:dla-length} and \eqref{claim:dla-twist} to bound 
$$\babs{\frac{d\la_1(g(t))}{dt}}\leq C \sum_{j\neq 1}\babs{\frac{\partial\la_1}{\partial\ell_j}}+2\pi
\sum_{j}\babs{\frac{\partial\la_1}{\partial\psi_j}}\leq C\ell^2$$ 
for every $t\in[0,1]$, where the constant $C$ depends only on $\hat \de$ and the genus. Integration over $[0,1]$ hence yields
the claimed bound of
$$\abs{\la_1(M,g)-f(\ell)}=\abs{\la_1(M,\tilde f^*g)-f(\ell)}=\abs{\la_1(g(1))-\la_1(g(0))}\leq C\ell^2 \text{ for some } C=C(\hat \de,\gamma).$$

\newpage
\subsection{Proof of Theorem~\ref{thm:sharp}}\label{subsec:proof_sharp} $ $ \\

We now turn to the proof that the estimates of Theorem~\ref{thm:FN} and Corollary~\ref{thm:C0-est} are sharp as claimed in Theorem~\ref{thm:sharp}. For this we proceed in two steps: First we show that suitable energy bounds, namely upper bounds on the energy on the thick part as obtained in Lemma~\ref{lemma:est-u-thick-main} and a \textit{lower} bound on the energy on a central part of a collar, can be turned into a \textit{lower} bound on the derivative of the eigenvalue with respect to the corresponding length coordinate. This is the purpose of Lemma~\ref{lemma:sharp} that we state and prove more generally for any simple eigenvalue of the Laplacian. The second step of the proof of Theorem~\ref{thm:sharp} is then to prove the necessary lower bounds on the energy of the first eigenfunction on the central part of a collar around a suitably short geodesic $\si^2$ (not homotopic to the given $\bar \si$), and this step is carried out by proving 
Lemma~\ref{lemma:sharp-genus3} for surfaces of genus at least $3$ respectively Lemma~\ref{lemma:sharp2} for surfaces of genus $2$.

\begin{lemma}\label{lemma:sharp}
Let $(M,g)$ be a closed hyperbolic surface, let $\de_0\in(0,\half\arsinh(1)]$ be any given number and let $\si^1,\ldots, \si^{j_0}$ be the simple closed geodesics in $(M,g)$ of length less then $2\de_0$ which we extend to a full collection of simple closed geodesics $\{\si^j\}_{j=1}^{3(\gamma-1)}$ that decompose $(M,g)$ into pairs of pants, chosen as always so that $L(\si^j)\leq \bar L=\bar L(\gamma)$.
\newline
Let $\la$ be any  simple eigenvalue of $-\Delta_g$ with normalised eigenfunction $u$ and denote by $\La$ an upper bound on $\norm{u}_{L^\infty(M,g)}$. 
Then there exists a universal constant $C_4>0$ and a constant $C_5=C_5(\de_0,\gamma,\La)$ so that the following hold true for any $j\in \{1,\ldots,j_0\}$:
The  derivative of $\la$ with respect to  $\ell_j$ is bounded from below by
\beq
\label{claim:lower-sharp}
\frac{\partial \la}{\partial \ell_j}\geq C_4 a_1\la^2-C_5(1+a_2)\ell_j\la^2
\eeq
where $a_1>0 $ is to be determined so that 
the energy of $u$ on the central part of the collar  $\Col(\si^j)$ is at least
\beq
\label{ass:s-energy-sharp}
\int_{-X(\ell_j)/2}^{X(\ell_j)/2}\int_{S^1} \abs{u_s}^2 d\th ds \geq a_1 \ell_j\la^2. 
\eeq
while $a_2>0$ is to be chosen so that
\beq 
\label{est:ass:-energy-est}
\norm{du}^2_{L^2(\de\thick(M,g))}\leq \frac{a_2}{\de}\la^2 \text{ for every } \de\in(0,\de_0].
\eeq

\end{lemma}

In this and the following lemmas we continue to use collar coordinates $(s,\th)$ on collars $\Col(\si)$ as described in Lemma~\ref{lemma:collar}, in particular $X(\ell)$ is given by \eqref{eq:rho-X}.

We first apply the above lemma to prove Theorem~\ref{thm:sharp} for surfaces of genus at least three, in which case we will want to consider surfaces $(M,g)$ which not only contain a disconnecting geodesic $\si^1$ of very small length $\ell_1$ but a further disconnecting geodesic  $\si^2$ whose length is quite small, but contained in a fixed interval. In this case we shall prove
\begin{lemma}\label{lemma:sharp-genus3}

For any genus $\gamma\geq 3$ and any number $ \de_0\in (0,\half\arsinh(1)]$ 
 there exist numbers $\bar\eta\in (0,\half \de_0)$ and $b_1>0$ depending only on $ \de_0$ and the genus of $M$ so that for any $\eta\in(0,\bar \eta]$ there exists $\bar\ell=\bar \ell(\eta,\de_0,\gamma)>0$ so that the following holds true:
 
 Let $(M,g)$ be a hyperbolic surface of genus $\gamma$ 
which contains two disconnecting simple closed geodesics $\si^1$ and $\si^2$ of length 
 $\ell_1\in (0,\bar\ell)$ and $\ell_2\in [2\eta,4\eta]$
 and for which furthermore $\inj_g \geq \de_0$ on $M\setminus ( \Col(\si^1)\cup \Col(\si^2))$.
 Then the energy of the normalised first eigenfunction $u_1$ on the central part of $\Col(\si^2)$ is bounded from below by
$$\int_{-X(\ell_2)/2}^{X(\ell_2)/2}\int_{S^1} \abs{\partial_s u_1}^2 d\th ds\geq b_1\ell_2^{-1} \la_1^2.$$
\end{lemma}

The above lemma hence implies in particular that the assumption \eqref{ass:s-energy-sharp} of Lemma~\ref{lemma:sharp} is satisfied on $\Col(\si^2)$ for a constant $a_1>0$ that depends only on $\de_0$ and the genus, while we note that \eqref{est:ass:-energy-est}
is satisfied for $a_2=C_0(\gamma)$, where $C_0$ is the constant obtained in Lemma~\ref{lemma:est-u-thick-main}. 
After reducing $\bar \eta$ if necessary, we hence obtain from Lemma~\ref{lemma:sharp} that for every $(M,g)$ as in Lemma~\ref{lemma:sharp-genus3} 
\begin{align}\label{eq:la1Thm14}\frac{\partial \la_1}{\partial \ell_2}\geq \half C_4 a_1\la_1^2\geq \tilde c \la_1^2 \text{ for some } \tilde c=\tilde c(\gamma,\de_0)>0.\end{align}
This establishes that the estimate on the derivatives of the first eigenfunction with respect to length coordinates $\ell_j$, $j\neq 1$, obtained in Theorem~\ref{thm:FN} is sharp for surfaces of genus at least $3$, while Lemma \ref{lemma:sharp2} will give the same result for surfaces of genus $2$.

We now explain how to use this bound to show that also the 
$C^0$-estimate 
\eqref{est:which_is_sharp} from Corollary~\ref{thm:C0-est} is sharp as claimed in Theorem~\ref{thm:sharp}. We carry out this proof in detail for surfaces of genus $3$, and remark that the same argument, now using Lemma \ref{lemma:sharp2} to obtain 
\eqref{eq:la1Thm14}, also yields the claim for surfaces of genus $2$.

We let 
$\hat \de:=\bar \eta(\gamma, \half\arsinh(1))\in (0, \frac14 \arsinh(1))$ and $\bar\ell=\bar\ell(\gamma,\half \arsinh(1))\in (0,\arsinh(1))$ be the numbers that we obtain above if we choose $\de_0=\half\arsinh(1)$.
We now construct two families of metrics $(g_\ell)_{\ell\in (0,\bar\ell)}$ and $(g_\ell)_{\ell\in (0,\bar\ell)}$ with the required properties as follows: 
Given any $\ell\in (0,\bar \ell)$ we let $(\hat g_\ell(t))_{t\in [0,1]}$ be the curve of metrics whose Fenchel-Nielsen coordinates are given by 
$$\ell_1(\hat g(t))=\ell, \quad 
\ell_2(\hat g(t))=2\eta+2\eta t, \text{ while } 
\ell_j(\hat g(t))=2\arsinh(1) \text{ for }j\geq 3 \text{ and } \psi_j(\hat g(t))=c_j
$$
for some fixed constant $c_j\in [0,2\pi]$ and $\eta=\bar{\eta}$. We will eventually consider $g_\ell=\hat g_\ell(0)$ and $\tilde g_\ell\define \hat g_\ell(1)$. 

We note that the argument used in the proof of Corollary~\ref{thm:C0-est} yields a lower bound on the 
 injectivity radius of $g_\ell(t)$ on 
 $M\setminus (\Col(\si^1(t))\cup \Col(\si^2(t)))$ of $\min(\arsinh(1),  \frac{w_{2\arsinh(1)}}{2})=\arsinh(1)> \de_0=\half\arsinh(1)$, as each $\ell_j\leq 2 \arsinh(1)$. 

We may thus apply Lemma \ref{lemma:sharp-genus3} and the resulting bound \eqref{eq:la1Thm14} for any of these metrics to conclude that
$$\frac{d\la_1(M,g_\ell(t))}{dt}=2\eta\frac{\partial \la_1}{\partial \ell_2}(g_\ell(t))\geq c\la_1(M,g_\ell(t))^2\text{ for every } t\in[0,1] \text{ and some } c=c(\gamma)>0.$$
As \eqref{est:la_rough} assures that $\la_1(M,g_\ell(t))\geq c\ell$, we can integrate this estimate over $[0,1]$  to conclude that the families $g_\ell=g_\ell(0)$ and $\tilde g_\ell=g_\ell(1)$ indeed have the required property that 
$$\la_1(M,\tilde g_\ell)-\la_1(M,g _\ell)\geq  c \min_{t\in[0,1]} \la_1(M,g_\ell(t))^2\geq \bar c\ell^2,$$ for a constant $\bar c=\bar c(\gamma)>0$.
This concludes the proof of Theorem~\ref{thm:sharp} for surfaces of genus $\gamma\geq 3$. 

We cannot apply Lemma \ref{lemma:sharp-genus3} if the genus of our surface is $\gamma=2$, as $M$ will not contain two disjoint disconnecting simple closed geodesics $\si^{1,2}$. For genus $2$ surfaces we instead consider a symmetric setting in which we can show

\begin{lemma}\label{lemma:sharp2}
There exist numbers $\bar \eta>0$ and $b_2>0$ so that for any $\eta\in(0,\bar\eta]$ there exists a number $\bar\ell=\bar \ell(\eta)>0$ so that the following holds true: Let $(M,g)$ be a hyperbolic surface  of genus $2$ 
which contains a disconnecting geodesic $\si^1$ of length
 $\ell_1\in (0,\bar\ell)$, and for which the other length coordinates agree and satisfy $\ell_2=\ell_3 \in [2\eta,4\eta]$, while all twist coordinates are zero.  Then the energy of the normalised first eigenfunction $u_1$ on the central part of $\Col(\si^{2,3})$ is bounded from below by
 \beq
\label{claim-ass:s-energy-sharp}
\int_{-X(\ell_j)/2}^{X(\ell_j)/2}\int_{S^1} \abs{\partial_s u_1}^2 dsd\th \geq b_2 \ell_j\la_1^2 \text{ for } j=2,3.
\eeq
\end{lemma}
It is important to note that for such symmetric surfaces the energy estimate \eqref{est:ass:-energy-est} also holds true for a constant $a_2$ that is independent of $\eta$, even though the assumptions of Lemma~\ref{lemma:est-u-thick-main} are violated if $\bar \de$ is chosen independently of $\eta$. Indeed, as we shall explain in detail in 
 Remark~\ref{rmk:symm-energy}, the proof of Lemma~\ref{lemma:est-u-thick-main} applies without change also for such symmetric surfaces containing short geodesics that are not disconnecting and yield that \eqref{est:ass:-energy-est} holds for $a_2=C$ for a universal constant $C$.

We may hence again apply Lemma~\ref{lemma:sharp} and the argument given above to conclude that our error rates on the first eigenvalue are sharp also for surfaces of genus $2$ as claimed in Theorem~\ref{thm:sharp}. 

It remains to give the proof of the above three lemmas and we begin with 

\begin{proof}[Proof of Lemma~\ref{lemma:sharp}]
Let $(M,g)$ be as in the lemma and let $\La^j$ the element which is dual to $d\ell_j$ as described in Definition \ref{defn:basis1}. 
As explained in Section~\ref{sec:dual} we know that 
if $\ell_j$ is small then $\La^j$ is concentrated essentially on the corresponding collar and there very close to $\Rea(b_0^j(\La^j))dz^2$ and furthermore recall that the principal parts of $\La^j$ are described by 
\eqref{est:princ-parts-La}.
 We may thus write 
\beqa  \label{est:proof-sharp-hallo}
\frac{\partial \la}{\partial \ell_j}& =
-\tfrac12 \langle \Rea(\Phi(u,g)), \Rea\La^j\rangle 
= \tfrac{\ell_j}{4\pi^2}\langle \Rea(\Phi(u,g)), \Rea(dz^2)\rangle_{L^2(\Col(\si^j))} +R,
\eeqa
for a remainder term that is bounded by 
\beqa \label{est:R-regen}
\abs{R}&\leq C\norm{\Phi}_{L^1(\de_0\thick(M,g))}\norm{\La^j}_
{L^\infty(\de_0\thick(M,g))} + \\
&\qquad + \sum_{k=1}^{j_0} \abs{\Ima(b_0^k(\La^j))}\cdot\abs{\Rea\langle \i dz^2, \Phi\rangle_{L^2(\Col(\si^k))} )}  +\sum_{k=1}^{j_0} \abs{\langle \La^j-b_0^k(\La^j)dz^2, \Phi\rangle_{L^2(\Col(\si^k))}}.
\eeqa

We now recall that the upper bound \eqref{est:ass:-energy-est} on the energy not only implies that $\norm{\Phi }_{L^1(\de_0\thick(M,g))}\leq Ca_2\la^2$, but furthermore 
allows us to apply 
Lemma~\ref{lemma:prod-Hopf}
to bound the above inner products of the Hopf-differential. 
We thus conclude that 
\beqa 
\abs{R}&\leq Ca_2\la^2 \norm{\La^j}_
{L^\infty(\de_0\thick(M,g))}  +C\sqrt{a_2} \la^2\sum_{k=1}^{j_0}\abs{\Ima(b_0^k(\La^j))} \ell_k^{-1}+C
a_2\la^2 \norm{\La^j}_{L^2(\frac{\de_0}{2}\thick(M,g))}\\
&\leq Ca_2 \ell_j\la^2+ C\sqrt{a_2} \ell_j\la^2\sum_{k=1}^{j_0}\ell_j\ell_k^3 \ell_k^{-1}\leq C(a_2+1)\ell_j\la^2
\eeqa
where we use the bounds \eqref{est:princ-parts-La} and \eqref{est:La-for-sharp}  on $\La^j$ from Section~\ref{sec:dual} in the penultimate step.

To bound the main term in \eqref{est:proof-sharp-hallo}, we can now use the angular energy estimate \eqref{est:weighted-ang-en-2}
as well as  \eqref{ass:s-energy-sharp}. Combined with the above bound on $R$ and
the fact that $\rho(\frac{X(\ell)}{2})\leq C\ell$ 
this yields 
\beqa
\frac{\partial \la}{\partial \ell_j}
&= \frac{\ell_j}{2\pi^2}\int_{-X(\ell_2)}^{X(\ell_2)}\int_{S^1}(\abs{u_s}^2-\abs{u_\th}^2)\rho^{-2} dsd\th +R \\
&\geq \frac{\ell_j}{2\pi^2}\rho^{-2}(\tfrac{X(\ell_2)}{2})\int_{-X(\ell_2)/2}^{X(\ell_2)/2}\int_{S^1}\abs{u_s}^2 d\th ds -C\ell_j \int_{-X(\ell_2)}^{X(\ell_2)}\int_{S^1} \abs{u_\th}^2 \rho^{-2}
 dsd\th +R\\
 &\geq \frac{\ell_j^2}{2\pi^2} a_1 \la^2 \rho^{-2}(\tfrac{X(\ell_2)}{2})-
 C\ell_j\norm{du}_{L^2(\de_3\thick (\Col(\si^j)))}^2-C\ell_j\la^2\norm{u}_{L^\infty(M,g)}^2+R\\
 &\geq C_4a_1\la^2-C_5 (a_2+1) \ell_j \la^2
\eeqa
for a universal constant $C_4>0$ and a constant $C_5$ that depends only on the genus, $\de_0$ and an upper bound on $\norm{u}_{L^\infty}$, as claimed in the lemma. 
\end{proof}

We now turn to the 
proofs of Lemmas \ref{lemma:sharp-genus3} and \ref{lemma:sharp2}. To this end we first show that in the setting of both of these lemmas 
there exist numbers $\bar \eta>0$ and $c_1>0$ that depend on the genus (and in the setting of  Lemma~\ref{lemma:sharp-genus3} additionally on $\de_0$ and there chosen in particular so that $\bar \eta<\de_0$) 
so that the following holds true:

For any 
$\eta\in (0,\bar\eta]$
there exists $\bar\ell=\bar\ell(\eta,\de_0,\gamma)>0$ so that if 
$(M,g)$ is as in Lemma~\ref{lemma:sharp-genus3} respectively \ref{lemma:sharp2}, then the normalised first eigenfunction $u_1$ is bounded away from zero pointwise on $\eta\thick(M,g)$. Namely, after replacing $u_1$ by $-u_1$ if necessary, we have that
\beq \label{est:pointwise-lower-bound}
u_1\geq c_1>0 \text{ on } M_1^\eta \text{ while } u_1\leq -c_1 \text{ on }  M_2^\eta
\eeq
where $M_{1,2}^\eta$ are the two connected components of $\eta\thick(M,g)=M\setminus \eta\thin(\Col(\si^1))$, with the convention that $\si^2\subset M_1^\eta$. 

We can prove this as follows: Let $(M,g)$ be a surface satisfying the assumptions of  Lemma~\ref{lemma:sharp-genus3} or \ref{lemma:sharp2} for some $\eta\in(0,\bar \eta)$, where $\bar \eta$ is determined below (in the setting of Lemma~\ref{lemma:sharp-genus3} chosen with $\bar \eta\leq \de_0$) and 
let $u_1$ be the normalised first eigenfunction of $-\Delta_g$.  We first note that 
\beq
\label{est:H2}
\norm{u_1}_{H^2(M,g)}^2\leq C\cdot(\norm{\Delta_g u_1}_{L^2(M,g)}^2+\norm{d u_1}_{L^2(M,g)}^2+\norm{u_1}_{L^2(M,g)}^2)\leq C\eeq 
is bounded by a constant that depends at most on the genus, as $\la_1$ is bounded from above uniformly.

%
We furthermore note that if we 
apply the Sobolev embedding theorem on  $M_{1,2}^\eta$ rather than on all of $M$, then the resulting estimate is valid with a constant that depends on $\eta$ and the genus, but not on $\ell_1$. Hence we obtain that  
\beq 
\label{est:osc} \osc_{M_{1,2}^\eta} u_1\leq C 
\norm{d u_1}_{L^3(M_{1,2}^\eta)}\leq C \norm{d u_1}_{L^4(M_{1,2}^\eta)}^{2/3}\cdot\norm{d u_1}_{L^2(M)}^{1/3} \leq C \cdot \la_1^{1/3}\eeq for a constant $C$ that is allowed to depend on 
$\eta$, $\de_0$ and the genus, but not on $\ell_1$. 
Given any $\eps>0$ and any $\eta$ we can hence choose 
 $\bar \ell$ sufficiently small (depending in particular on $\eta$ and $\eps$) so that the above estimate ensures that 
$\osc_{M_{1,2}^\eta} u_1\leq \eps$ provided $\ell_1\leq \bar \ell$ as assumed in the lemmas.  

In particular, for $\bar u_{1,2}:=\fint_{M^{\eta}_{1,2}}u_1 dv_g$
we have that $\abs{u_1-\bar u_{1,2}}\leq \eps$ on $M_{1,2}^\eta$ and hence
\beqs 
(\abs{\bar u_1}+\eps)^2\Area(M_1^\eta)+(\abs{\bar u_2}+\eps)^2\Area(M_2^\eta)\geq \norm{u_1}_{L^2(M,g)}^2-\norm{u_1}_{L^\infty}^2\Area(\eta\thin(M,g))\geq 1-C\bar\eta
,\eeqs
where $C$ depends only on an upper bound on $\norm{u_1}_{L^\infty}$ and is in particular independent of $\bar\eta$. At the same time the fact that $\int_M u_1=0$ implies $$\bar u_1\Area(M_1^\eta)+\bar u_2\Area(M_1^\eta)\leq \norm{u_1}_{L^\infty}\Area(\eta\thin(M,g))\leq C\bar \eta.$$
If $\bar \eta>0$ is initially chosen small enough, we thus find 
that to every $\eta\in (0,\bar \eta)$ there exist numbers $\eps>0$ and $\bar\ell>0$ so that the above argument ensures that 
$\bar u_{1,2}$ have the opposite sign and satisfy  $\abs{\bar u_{1,2}}\geq 2c_1>c_1+\eps$ for some $c_1=c_1(\gamma,\de_0)>0$. As $\osc_{M_{1,2}^\eta} u_1\leq \eps$ we thus obtain the desired pointwise bound \eqref{est:pointwise-lower-bound}.
 Based on this bound on $u_1\vert_{M_{1,2}^\eta}$ we can now complete the proofs of Lemmas \ref{lemma:sharp-genus3} and \ref{lemma:sharp2}  as follows. 

\begin{proof}[Proof of Lemma~\ref{lemma:sharp-genus3}]
Let $u_1$ be as in the lemma, let $\bar \eta$ be as above and let 
$\mu\geq 0$ be so that 
\beq\label{eq:def-mu}
\int_{-X(\ell_2)/2}^{X(\ell_2)/2}\int_{S^1} \abs{\partial_s u_1}^2= \mu\ell_2^{-1} \la_1^2.\eeq
Our goal is to derive a lower bound on $\mu$ that depends only on the genus $\gamma\geq 3$ and $\de_0$.

To this end, we 
compare the Rayleigh-quotient of $u_1$ with the 
one of
 $u_1^\epsilon=u_1+\eps\cdot (v-\fint_M v)$, where $v$ is chosen to be 
linear in the collar coordinate $s$ on the set $\Col_{1/2}(\si^2)=\{ (s,\th): \abs{s}\leq \half X(\ell_2)\}\subset \Col(\si^2)$ with 
$v\equiv 0$ on the connected component $F^-$ of  $M\setminus \Col_{1/2}(\si^2)$
which contains $\si^1$ and $v\equiv 1$ on the other connected component $F^+$ of $M\setminus \Col_{1/2}(\si^2)$. 

We note that the support of $v$ is contained in $M_1^\eta$, where the eigenfunction $u_1\geq  c_1>0$, so as $\int_M u_1=0$ we have that at $\eps=0$
\beqa
\frac{d}{d\eps}\norm{u_1^\epsilon}_{L^2(M,g)}^2&=2\int_M u_1 v dv_g-2\int_M u_1 dv_g\cdot \fint_M v dv_g= 2\int_M u_1 v dv_g \geq c_1 \Area(F^+) \geq c_2>0\eeqa
for a constant $c_2>0$ that depends only on the genus of $M$ and the fixed number $\de_0$.  

Conversely, the change of the energy of $u_1^\epsilon$ at $\eps=0$ can be no more than 
\beqa 
\frac{d}{d\eps} \norm{d u_1^\epsilon}_{L^2(M,g))}^2&=2\int_{-X(\ell_2)/2}^{X(\ell_2)/2}\int_{S^1} \partial_s v \partial_s u_1 ds d\th\\
& \leq C X(\ell_2)^{-1/2} \cdot \bigg(\int_{-X(\ell_2)/2}^{X(\ell_2)/2}\int_{S^1}\abs{\partial_s u_1}^2\bigg)^{1/2}
\leq C \mu^{1/2} \la_1
\eeqa 
for the number $\mu$ from \eqref{eq:def-mu} that we want to bound from below and a universal constant $C$. 

Since $u$ is a critical point of the Rayleigh-quotient we thus know that at $\eps=0$
$$0=\frac{d}{d\eps} \norm{d u_1^\epsilon}_{L^2(M,g))}^2-\la_1 \frac{d}{d\eps}\norm{u_1^\epsilon}_{L^2(M,g)}^2\leq C \mu^{1/2} \la_1-c_2\la_1 .$$
This gives the uniform lower bound of $\mu\geq \big(\frac{c_2}{C}\big)^2=:b_1$ claimed in the lemma. 
\end{proof}

\begin{proof}[Proof of Lemma~\ref{lemma:sharp2}]
We argue very similarly as in the previous proof, with the main difference being that we no longer have that $\si^2$ is disconnecting. 
We hence choose $v$ to be $v\equiv 1$ on the subset of points in $\Col(\si^2)$ whose collar coordinates satisfy $\abs{s}\leq \frac14 X(\ell)$, $v\equiv 0$ on $M\setminus \Col_{1/2}(\si^2)$ and $v$ linear on the cylinders that connect these two parts of the collar, i.e. for 
$\frac14 X(\ell_2)\leq \abs{s}\leq \frac12 X(\ell_2)$. 
The area of the set on which $v$ is identically one is 
then bounded from below by $ c\ell_2$ for some universal $c>0$, so that we now have that for $u_1^\epsilon$ as in the previous proof 
$$\frac{d}{d\eps}\norm{u_1^\epsilon}_{L^2(M,g)}^2 \geq c\ell_2>0.$$

Letting this time $\mu\geq 0$ be so that 
$\int_{-X(\ell_j)/2}^{X(\ell_j)/2}\int_{S^1} \abs{\partial_s u_1}^2 dsd\th = \mu \ell_2\la_1^2 $, we can then bound
$$
\frac{d}{d\eps} \norm{d u_1^\epsilon}_{L^2(M,g)}^2\leq C \mu^{1/2} \la_1 \ell_2
$$ so the claimed uniform bound on $\mu$ again follows from the fact that $u_1$ is a critical point of the Rayleigh-quotient and hence $C \mu^{1/2} \la_1 \ell_2\geq c\ell_2\la_1$.  
\end{proof}

\section{Proof of the properties of the first eigenfunction}
\label{sect:proof_ef}
In this section we prove the energy estimates on the first eigenfunction that we stated in Section~\ref{subsec:energy} and used in the previous sections to prove our main results about the first eigenvalue. Some of these proofs are based on a specific form of the Poincar\'e estimate that we state in Lemma~\ref{lemma:Poincare-functions} below, and for which we include a proof in Section~\ref{subsection:last}, where we also provide a proof of Remark~\ref{rmk:L-infty}.

\subsection{Proof of the energy estimates of the first eigenfunction } $ $\\
To begin with, we give a brief sketch of the proof of the angular energy estimates stated in Lemma~\ref{lemma:ang-energy-mainpart}, which follow from very well-known arguments that have been used in particular in many works in the analysis of bubbling for harmonic maps; a very similar proof can be found e.g. in \cite[Lemma 2.4]{HRT}
\begin{proof} [Sketch of proof of Lemma~\ref{lemma:ang-energy-mainpart}]
Let $(M,g)$ be a closed hyperbolic surface and let $\Col(\si)\subset (M,g)$ be a collar around a geodesic $\si$ of length $\ell\in (0,2\arsinh(1))$ on which we introduce collar coordinates $(s,\th)$. Let $u$ be any normalised eigenfunction of $-\Delta_g $ to an eigenvalue $\la$ and let 
 $\vartheta(s)\define \int_{\{s\}\times S^1}\abs{u_\th}^2 d\th$. 
 A short calculation shows that 
$$
\vartheta''(s)-\vartheta(s)
\geq -\int_{\{s\}\times S^1} \rho^4\abs{\Delta_g u}^2 d\theta=-
\la^2\int_{\{s\}\times S^1} \rho^4u^2 d\th.
$$
Hence, comparison with the solution of the corresponding ODE implies that for
any $\La\geq 0$ and any $s$ with 
 $\abs{s}\leq X(\ell)-(\La+1)$ we have
\beqa \label{est:theta-with-Lambda}
\vartheta(s) &\leq  2\cdot e^{-\La}\norm{d u}_{L^2(\Cyl_{\abs{s}+\La+1}\setminus \Cyl_{\abs{s}+\La})}^2 +\half \la^2 \int_{\Cyl_{\abs{s}+\La+1}} \rho^4e^{-\abs{s-q}} u^2 d\th dq,
\eeqa
where we write for short $\Cyl_\La=[-\La,\La]\times S^1$. In particular, we obtain 
that, for some universal constants $\de_3,C>0$, we can bound 
\beqa \label{est:ang-en-1}
\vartheta(s)\leq  C  e^{-(X(\ell)-\abs{s})}
\norm{du}_{L^2(\de_3\thick(\Col(\si)))}^2+C\la^2\int_{-X(\ell)}^{X(\ell)}\int_{S^1} \rho^4 e^{-\abs{s-q}}u^2 d\th dq
\eeqa
for every $s\in [-X(\ell)+1,X(\ell)-1]$. A short calculation, integrating the above estimate with the desired weight of $\rho^{-\alpha}$, $\alpha=2,4$,  
using Fubini's theorem and the fact that $\abs{\partial_s\rho^{-1}}\leq 1$, then yields the desired bounds \eqref{est:weighted-ang-en-4} and  \eqref{est:weighted-ang-en-2} on $\int\rho^{-\alpha} \vartheta$, first for the integral over $\abs{s}\leq X(\ell)-1$, but as $\rho^{-1}$ is bounded uniformly near the ends of the collars, hence also for the integral over the whole collar. 
\end{proof}

We now turn to the proof of the estimates on the energy of the first eigenfunction on the thick part of the surface that we stated in Lemma~\ref{lemma:est-u-thick-main}. For this we shall use the following version of the Poincar\'e inequality for functions, a proof of which is included in the next section for the sake of completeness.

\begin{lemma}
\label{lemma:Poincare-functions}
Let $(M,g)$ be a closed oriented hyperbolic surface and suppose that $\de\in (0,\half\arsinh(1))$ is so that $\inj(M,g)\leq \de$. 
Let $M_1^\de$ be the closure of a connected component of $\{p: \inj_g(p)> \de\}$ and denote its boundary components by  
$\partial M_1^\de= \gamma^1\sqcup \ldots \sqcup \gamma^{k_1}$. 
Suppose that  $v\in H^1(M,g)$ 
is so that 
\beqs 
\label{ass:poincare-fn}
v\equiv 0 \text{ on at least one } \gamma^j.\eeqs
Then we may estimate
\beqs \label{est:Poincare-fn}
\norm{v}_{L^2(M_1^\de,g)}^2+ \de^{-1} \norm{v}_{L^2(\partial M_1^\de,g)}^2
\leq \frac{C}{\de}\norm{d v}_{L^2(M_1^\de,g)}^2
\eeqs
for a constant $C$ that depends only on the genus of $M$.
\end{lemma}

\begin{proof}[Proof of Lemma~\ref{lemma:est-u-thick-main}]
Let $(M,g)$ be as in the lemma and let $\La_0=\La_0(\gamma)\geq 0$ be a fixed number that is chosen below. We recall that points in
the $\de\thin$ part of a collar $\Col(\si)$, $\ell=L_g(\si)< 2\de\leq 2\arsinh(1)$, have collar coordinates with  $\abs{s}< X_\de(\ell)$,
for $X_\de(\ell)$ given by \eqref{eq:Xde}. Hence we can and will choose 
 $\de_2=\de_2(\gamma)\in (0,\arsinh(1))$ sufficiently small so that 
$X(\ell)-X_\de(\ell)\geq \La_0+2$ for every $\ell\leq 2\de_2$ and $\de\in (\half \ell,\de_2]$. 

We also note that if the assumptions of the lemma are satisfied for 
$\bar \de$ then they are also satisfied if we replace $\bar \de$ by 
$\min(\bar\de,\de_2)$, and remark that proving \eqref{est:u-de-thick} for $\de<\min(\bar\de,\de_2)$ also gives the desired bounds for $\de\in [\de_2,\bar\de)$ since  $\de_2$ depends only on the genus. From here on we thus assume without loss of generality that $\bar \de\leq \de_2$. 

In addition, it suffices to consider surfaces for which $\inj(M,g) < \bar\de$, 
as otherwise \eqref{est:u-de-thick} is trivially satisfied since $\la_1$ would be bounded away from zero in terms of $\bar\de$ and the genus. It furthermore suffices to consider  numbers 
 $\de\geq \inj(M,g)$, as the estimate for smaller values of $\de$ follows from the case that $\de=\inj(M,g)$.

So let $\de\in [\inj(M,g),\bar\de]$ where $\bar\de\leq \de_2<\arsinh(1)$ is as in the lemma. We first note that the assumptions of the lemma guarantee 
 that the set of simple closed geodesics $\{\si^{i}\}_{i=1}^{k}$ of length no more than $2 \de$ is non-empty and contains only geodesics $\si^i$ for which  $M\setminus \si^i$ is disconnected. Since the length of these geodesics is less than $2\arsinh(1)$, the $\si^j$ are furthermore pairwise disjoint so  $M\setminus \bigcup_{i=1}^k\si^i $ has
$k+1$ connected components which we denote by $M_i$. We furthermore note that by construction 
$M_i^\de:=\de\thick(M,g)\cap \ov{M_i}$ are the closures of the connected components of 
$\{p:\inj_g(p)> \de\}$ and remark that $\bigcup M_i^\de=\de\thick(M,g)$ as well as that 
$\de\thin(M,g)\subset \bigcup_{i=1}^k \Col(\si^i)$.

The basic idea of the proof is the following: If too much energy was concentrated on one of the $M_i^\de$, then a function which is constant on (most of) 
$M_i^\de$, but agrees with $u$ up to a constant on each of the connected components of $M\setminus M_i^\de$, would have a smaller Rayleigh-quotient than $u_1$, contradicting the fact that $u_1$ is a first eigenfunction. 
To make this idea precise, we associate to each
$M_i^\de$ the numbers $\mu_i=\mu_i(\de)$ which are determined by 
\beq 
\label{def:mu}
\norm{d u_1}_{L^2(M_i^\de,g)}^2=\mu_i\frac{\la_1^2}{\de}, \qquad i=1,\ldots,k+1. \eeq
After reordering we may assume without loss of generality that $\mu_i\leq \mu_1$, $i=2,\ldots,k+1$, so to establish the claim of the lemma we need to show that 
$\mu_1\leq C$ for a constant $C$ that depends only on the genus.

Let $\gamma^1,\ldots,\gamma^{k_1}$ be the boundary curves of $M_1^\de$. As the injectivity radius is equal to $\de$ on $\partial M_1^\de$, each $\gamma^i$ must lie in a collar around a geodesic $\si^{j_i}$ of the collection of $\{\si^j\}$ obtained above. 
The assumption that $M\setminus \si^j$ is disconnected 
for each $j$ ensures that $j_i\neq j_k$ for $i\neq k$ as well as that each of the connected components $\widetilde M^i$ of $M\setminus M_1^\de$ is adjacent to precisely one $\gamma^i$. 
We may thus assume without loss of generality that $\gamma^i$ is contained in the closure of $\Col(\si^i)\cap \widetilde M^i$.
In collar coordinates (chosen with suitable orientation) $\gamma^i$ then corresponds to the curve $\{X_\de(\ell_i)\}\times S^1$, $\ell_i=L_g(\si^i)$, while $M_1^\de\cap \Col(\si^i)$ corresponds to the cylinder $[X_\de(\ell_i),X(\ell_i))\times S^1$. We recall that the choice of $\de_2$ made above guarantees that $X(\ell_i)-X_\de(\ell_i)\geq \La_0+2\geq 2$.

We will later consider the Rayleigh-quotient of  $v=\tilde u-\fint_M \tilde u$ where $\tilde u\in H^1(M,g)\cap C^0(M,g)$ is obtained
as modification of $u_1$ as follows:
We let 
$c_i\define\fint_{\gamma^i} u_1 dS_g=(2\pi)^{-1}\int_{\{X_\de(\ell_i)\}\times S^1} u_1 d\th$
and define $\tilde u$ so that $\tilde u-u_1$ is constant on each  connected component $\widetilde M^j$ of $M\setminus M_1^\de$
while  $\tilde u\equiv c_1$ 
 on all of $M_1^\de$ except for the cylinders 
 $K_j=[X_\de(\ell_j),X_\de(\ell_j)+1]\times S^1 \subset \Col(\si^j)$ on which we interpolate. To be more precise, we set 
\beq\label{eq:def-tildeu}
\tilde u(p)= \begin{cases}
c_1 & \text{ for } p\in M_1^\de\setminus \bigcup_{j=1}^{k_1}K_j
\\
c_1+(X_\de(\ell_j)+1-s)\cdot\big[u_1(X_{\de}(\ell_j),\th)-c_j\big]
 & \text{ for }  p=(s,\th) \in  K_j\subset \Col(\si^j)\\
u_1+(c_1-c_j)  & \text{ for } p\in \widetilde M^j.
\end{cases}
\eeq
We first note that by \eqref{def:mu}
\beq \label{est:energy_diff_tildeu1}
\norm{du_1}_{L^2(M,g)}^2-\norm{ d\tilde u}_{L^2(M,g)}^2=\norm{du_1}_{L^2(M_1^\de,g)}^2-\sum_j\norm{d\tilde u}_{L^2(K_j,g)}^2=\mu_1\frac{\la_1^2}{\de}-\sum_j\norm{d\tilde u}_{L^2(K_j,g)}^2.
\eeq
The choice of $\de_2\geq \de$ guarantees that 
$\abs{s}\leq X(\ell_j)-(\La_0+1)$ on 
the cylinders $K_j$ on which we interpolate, so 
we may apply the angular energy estimate \eqref{est:theta-with-Lambda} with $\La=\La_0\geq 0$ to obtain 
\beqas
\norm{d \tilde u}_{L^2(K_j,g)}^2&\leq\int_{\{X_\de(\ell_j)\}\times S^1}\abs{u_1-c_j}^ 2+\abs{\partial_\th u_1}^2  d\th \leq 2\vth(X_\de(\ell_j))\\
&\leq 
4e^{-\La_0}\norm{d u_1}_{L^2(\Cyl_{X_\de(\ell_j)+\La_0+1}\setminus \Cyl_{X_\de(\ell_j)+\La_0})}^2 + \la_1^2 
\int_{\Cyl_{X_\de(\ell_j)+\La_0+1}}
 \rho^4e^{-\abs{ X_\de(\ell_j)-q}} u_1^2 d\th dq \\
&\leq 4 e^{-\La_0} \norm{d u_1}_{L^2(\Col(\si^j)\cap \de\thick(M,g))}^2+C\la_1^2\rho^4(X_\de(\ell_j)+\La_0+1)\norm{u_1}_{L^\infty(M,g)}^2.
\label{est:energy-tilde-u-3}
\eeqas
We finally choose 
$\La_0$ so that $k+1\leq 3(\gamma-1)+1 \leq \frac{1}{8}e^{\La_0}$
where we recall that 
$k+1$ is the number of connected components of $\{p:\inj_g(p)> \de\}$. We also recall from  \eqref{est:inj-by-rho} that 
$\rho(X_\de(\ell_j)+\La_0+1)\leq e^{\La_0+1} \rho(X_\de(\ell_j))\leq  C \de$
and from Remark~\ref{rmk:L-infty}
that $u_1$ is uniformly bounded. We hence obtain 
\beqa \label{est:energy-tilde-u-4}
\sum_{i=1}^{k+1} \norm{d \tilde u}_{L^2( K_i,g)}^2 &\leq 4 e^{-\La_0} \norm{d u_1}_{L^2(\de\thick(M,g))}^2+C\la_1^2\de^4
= 4 e^{-\La_0}\sum_{i=1}^{k+1} \mu_i \cdot \frac{\la_1^2}{\de} +C\la_1^2\de^4
\\&\leq 4(k+1)e^{-\La_0} \mu_1\frac{\la_1^2}{\de}+C\la_1^2\de^4
\leq \half \mu_1\frac{\la_1^2}{\de}+C\la_1^2\de^4,
\eeqa
where the numbers $\mu_i$ are as in \eqref{def:mu} and where we used $\mu_i\leq \mu_1$ in the penultimate step.

On the one hand, we can combine this estimate with 
 \eqref{est:energy_diff_tildeu1} to obtain that 
\beqa
\label{est:energy-tilde-u}
\norm{d \tilde u}_{L^2(M,g)}^2\leq  &
\norm{d u_1}_{L^2(M,g)}^2 -\half \mu_1\frac{\la_1^2}{\de}
+C\la_1^2\de^4
=\la_1\cdot \big[1-(\frac{\mu_1 \la_1}{2\de}-C\la_1 \de^4)\big]
\eeqa
for a constant $C$ that depends only on the genus. 

On the other hand, \eqref{est:energy-tilde-u-4} 
 also allows us to estimate the $L^2$-norm of $u_1-\tilde u$: 
Since  $u_1-\tilde u\equiv 0$ on $\gamma^1$ and 
\beqs 
\norm{d(u_1-\tilde u)}_{L^2(M_1^\de,g)}^2\leq 2\norm{du_1}_{L^2(M_1^\de,g)}^2+2\norm{d\tilde u}_{L^2(M_1^\de,g)}^2\leq C\mu_1\frac{\la_1^2}{\de}
+C\la_1^2\de^4
\eeqs
we may apply the Poincar\'e estimate stated in  Lemma~\ref{lemma:Poincare-functions} to bound
\beqa \label{est:L2-diff-u-tilde}
\norm{u_1-\tilde u}_{L^2(M_1^\de,g)}^2&\leq C\mu_1\la_1^2\de^{-2}+C\de^3\la_1^2.
\eeqa
The same argument, now using the trace-estimate from  Lemma~\ref{lemma:Poincare-functions}, implies that also
\beqa\label{est:L2-diff-u-tilde_2}
\norm{u_1-\tilde u}_{L^2(M\setminus M_1^\de,g)}^2&=\sum_i \abs{c_i-c_1}^2 \Area(\widetilde M^i,g) \leq C \sum_i \int_{\{X_\de(\ell_i)\}\times S^1} \abs{u_1-\tilde u}^2 d\theta\\
&=C\delta^{-1} \norm{ u_1-\tilde u}_{L^2(\partial M_1^\delta,g)}^2 
\leq \frac{C}{\de}\norm{d(u_1-\tilde u)}_{L^2(M_1^\de,g)}^2\\
&\leq 
 C\mu_1\la_1^2\de^{-2}+C\de^3\la_1^2.\eeqa
 We now set $v=\tilde u-(\tilde u)_M$  where we write for short 
$(\tilde u)_M:=\fint_M \tilde u dv_g$ and note that since 
$\int u_1 dv_g=0$ we have 
\beqas 
\label{est:mv-tilde-u}
\abs{(\tilde u )_M}\leq C\norm{u_1-\tilde u}_{L^2(M,g)} .\eeqas
Since $u$ and hence also $v$ is bounded uniformly and since $\norm{u_1}_{L^2(M,g)}=1$ we thus get
\beqas \label{est:v-l2-1}
1-\norm{v}_{L^2(M,g)}^2&\leq \norm{u_1+v}_{L^2(M,g)}\cdot \norm{u_1-v}_{L^2(M,g)}\leq C(\norm{u_1-\tilde u}_{L^2(M,g)}+\norm{(\tilde u)_M}_{L^2(M,g)})\\
&\leq C\norm{u_1-\tilde u}_{L^2(M,g)}.
\eeqas
Inserting \eqref{est:L2-diff-u-tilde} and \eqref{est:L2-diff-u-tilde_2} into this estimate hence gives a bound of 
\beqs
\norm{v}_{L^2(M,g)}^2\geq 
1- C[\mu_1^{1/2}\la_1\de^{-1}+\de^{3/2}\la_1].
\eeqs
We may finally combine this estimate with the bound \eqref{est:energy-tilde-u} on $\norm{d\tilde u}_{L^2}^2=\norm{dv}_{L^2}^2$ to reach
\beqas
\frac{\norm{d v}_{L^2(M,g)}^2}{\norm{v}_{L^2(M,g)}^2}&\leq \la_1 \cdot \frac{1-(\frac{\mu_1 \la_1}{2\de}-C\la_1 \de^4)}{1-C (\de^{-1}\mu_1^{1/2} \la_1+\de^{3/2}\la_1)}.
 \eeqas
 Since this quotient can be no smaller than the first eigenvalue $\la_1$ we must thus have that 
$$\mu_1 \leq C\cdot \frac{2\de}{\la_1}\cdot [ \la_1 \de^4+\de^{-1}\mu_1^{1/2} \la_1+\de^{3/2}\la_1 ]\leq \half \mu_1 +C(\de_2^5+\de_2^{5/2}+1),$$
which gives the desired uniform upper bound on $\mu_1$ and hence completes the proof of the lemma. 
\end{proof}

We finally explain why the above proof still applies in the setting of Lemma~\ref{lemma:sharp2} where we have two quite short geodesics $\si_{2,3}$ which are not disconnecting, and hence why also in this setting the energy estimate \eqref{est:ass:-energy-est} holds true for a constant that is independent of $\ell_{2,3}$ as required in the proof of Theorem~\ref{thm:sharp}.  

\begin{rmk}\label{rmk:symm-energy}
Let $(M,g)$ be a genus two surface with Fenchel-Nielsen coordinates $\ell_2=\ell_3\in [2\eta,4\eta]$, $\ell_1\in (0,\bar\ell=\bar\ell(\eta))$, and $\psi_{1,2,3}=0$ as considered in Lemma~\ref{lemma:sharp2}. 
We note that in this situation, the assumptions of Lemma~\ref{lemma:est-u-thick-main} are not satisfied if we choose $\bar \de$ to be independent of $\eta$, say $\bar \de=\arsinh(1)$. On the other hand, if we drop the assumption that all geodesics of length no more than $2\bar \de$ are disconnecting then we can have that two of the boundary curves of the same connected component of $ \{p:\inj_p(g)>\de\}$ are contained in the same collar around a geodesic $\tilde \si$ which is not disconnecting. We may however still apply the above proof  in such a situation provided we know that
the mean values of 
$u_1\vert_{\Col(\tilde \si)}$ over these two curves $\{-X_\de(\tilde \ell)\}\times S^1$ and $\{X_\de(\tilde \ell)\}\times S^1$, 
 used in the definition of $\tilde u$, agree. 
In the setting of Lemma~\ref{lemma:sharp2} this is indeed the case: 
As we consider only values of $\ell_1$ for which $\la_1$ is simple and for which \eqref{est:pointwise-lower-bound} holds, we know that the restrictions
of the eigenfunction to the two collars $\Col(\si^{2,3})$ must be even, i.e. $u_1(s,\theta)=u_1(-s,\theta)$ in the corresponding collar coordinates on $\Col(\si^{2,3})$. In particular, the above mentioned mean values agree and \eqref{eq:def-tildeu} still gives a well-defined comparison function $\tilde u$.
\end{rmk}

\subsection{Proof of Lemma~\ref{lemma:Poincare-functions} and Remark~\ref{rmk:L-infty}}\label{subsection:last}$ $\\
For the sake of completeness we finally provide a proof of the Poincar\'e estimate used in the previous part, as well as a proof of the uniform $L^\infty$-bound on $u_1$ used throughout the paper.

\begin{proof}[Proof of Lemma~\ref{lemma:Poincare-functions}]
Let $(M,g)$ be as in Lemma~\ref{lemma:Poincare-functions} and let $\de_0:=\half\arsinh(1)$. 
We note that the diameter of the connected components $\widehat M^j$ of $\de_0\thick(M_1^\de,g)$ is bounded from above in terms of the genus of $M$, so standard versions of the Poincar\'e inequality, combined with the trace-theorem, imply that 
\beqa 
\Vert v\Vert_{L^2(\widehat{M}^{j},g)}^2
+\max_i\Vert v\Vert_{L^2(\widehat \gamma_i^j, g)}^2 
\leq C\left( \min_i\Vert v\Vert_{L^2(\widehat \gamma_i^j, g)}^2 +  \Vert d v\Vert_{L^2(\widehat{M}^{j},g)}^2\right)\label{est:Poinc_thick},
\eeqa
$\widehat \gamma_i^j$ the boundary curves of $\widehat{M}^{j}$ and $C$ a constant that depends only on the genus.

The connected components $K^j$ of 
$\overline{\de_0\thin(M_1^\de,g)}$ are now given by hyperbolic cylinders which
are subsets of collars around geodesics $\si^j$ of length $\ell_j< 2\de_0$. 
These cylinders are described  in 
collar coordinates by 
$ K^j= [X_j^-,X_j^+] \times S^1$ where $X_j^\pm$ are as follows: 
If 
 $\ell_j\geq 2\de$, then $K^j$ does not contain a boundary curve of $M_1^\de$ and hence  $X_j^+= -X_j^-= X_{\de_0}(\ell_j)$, where $X_\de(\ell)$ is described in  \eqref{eq:Xde}. Conversely if $\ell_j<\de$ then one of the boundary curves of $K^j$ coincides with a boundary curve of $M_1^\de$ and hence 
  (after changing the orientation of the collar coordinates if necessary) $X_j^-
=X_\de(\ell_j)$ while $X_j^+=X_{\de_0}(\ell_j)$. 
In both cases, the (euclidean) length of these cylinders is bounded from above by $C\de^{-1}$ thanks to  \eqref{eq:Xde}. Hence, a short calculation shows that 
\beqa \label{est:p1} 
\Vert v\Vert^2_{L^2(K^j,g)} +\max_{\pm} \int_{\{X_j^\pm\}\times S^1} v^2 d\th\leq C\min_{\pm} \int_{\{X_j^\pm\}\times S^1} v^2 d\th + \frac{C}{\de} \Vert dv\Vert^2_{L^2(K^j,g)}  
\eeqa
for a universal constant $C$. 

We note that  the $L^2$-norms of the traces on circles $\gamma=\{s\}\times S^1$ with respect to the euclidean and hyperbolic metric are related by 
$\norm{v}_{L^2(\gamma,g)}^2=\rho(s)\int_{\{s\}\times S^1}v^2 d\th$. As $\rho(s)$ scales like the injectivity radius, we thus have that the integrals over circles $\gamma=\{X_j^\pm\}\times S^1$ appearing in the above formula are  of order $\frac{C}{\de} \norm{v}_{L^2(\gamma,g)}$ if $\gamma$ coincides with one of the boundary curves $\gamma^i$ of $M_1^\de$.

The lemma now follows by iteratively applying  these two estimates \eqref{est:Poinc_thick} and \eqref{est:p1}
 on adjacent connected components of $\de_0\thick(M_1^\de,g)$ and $\de_0\thin(M_1^\de,g)$, starting with the component that contains the boundary curve $\gamma^j$ on which $v$ vanishes.
\end{proof}

\begin{proof}[Proof of Remark~\ref{rmk:L-infty}]
Let $(M,g)$ be an oriented hyperbolic surface for which the shortest closed geodesic $\si^1$ is disconnecting. 
We let $\de>0$ be a universal constant that is chosen small enough so that $\de\thin(M,g)$ is contained in the subsets $\Cyl_{X(\ell_j)-1}=\{(s,\th): \abs{s}\leq X(\ell_j)-1\}$ of the collars $\Col(\si^j)$ around the simple closed geodesics $\si^j$ of length $\ell_j<2\arsinh(1)$. As observed in the proofs of Lemmas \ref{lemma:sharp-genus3} and \ref{lemma:sharp2}, see in particular \eqref{est:H2} and \eqref{est:osc}, we have a uniform bound on the $H^2$-norm of $u_1$  and hence obtain a uniform bound on the oscillation 
 of $u_1$ over each connected component of  $\de\thick(M,g)$.
To bound the oscillation of $u_1$ over subsets of the collars $\Col(\si^j)$ considered above, we first recall that the angular energy estimate \eqref{est:theta-with-Lambda} gives a uniform upper bound on the oscillation over circles $\{s\}\times S^1$ in this set. 
At the same time, we can bound 
\begin{align*} 
\abs{\fint_{\{s_1\}\times S^1}u_1 d\theta-\fint_{\{s_2\}\times S^1}u_1 d\theta}\leq C \int_{\Col(\si^j)} \abs{\partial_s u_1} dsd\th= C\cdot \la_1^{1/2} X(\ell_j)^{1/2}\leq  C\la_1^{1/2} \ell_j^{-1/2}\leq C.
\end{align*}
for any 
$s_{1,2}\in [-X(\ell_j)+1,X(\ell_j)-1]$ and any $j$.
Combined we thus obtain a uniform bound on $\osc_M u_1$ and so, as $\int_Mu_1=0$, on $\norm{u_1}_{L^\infty(M,g)}$ as claimed. 
\end{proof}

{\sc Mathematisches Institut, Universit{\"a}t Freiburg, 79104 Freiburg, Germany }\\
{\sc Mathematical Institute, University of Oxford, Oxford, OX2 6GG, UK}

\end{document}